\documentclass{amsart}
\usepackage{graphicx} 
\usepackage{amsmath}
\usepackage{bm}
\usepackage{amsfonts}
\usepackage{amssymb}
\usepackage{color}
\usepackage{mathtools}
\usepackage{amsthm}
\usepackage{dsfont}
\usepackage{enumitem}
\usepackage[numbers]{natbib}

\usepackage[
  a4paper,
  top=40mm,
  bottom=40mm,
  left=30mm,
  right=30mm
]{geometry}

\usepackage{setspace}
\newcommand\bR{{\mathbb{R}}}

\newcommand\bN{{\mathbb{N}}}

\newcommand\cA{{\mathcal{A}}} 
\newcommand\cC{{\mathcal{C}}}

\newcommand\cL{{\mathcal{L}}}
\newcommand\cF{{\mathcal{F}}}
\newcommand\cK{{\mathcal{K}}}
\newcommand\cT{{\mathcal{T}}}
\newcommand\cP{{\mathcal{P}}}

\newcommand\cN{{\mathcal{N}}}

\newcommand{\eps}{\varepsilon}

\allowdisplaybreaks 

\theoremstyle{plain}
\newtheorem{thm}{Theorem}
\newtheorem{lem}[thm]{Lemma}
\newtheorem{prop}[thm]{Proposition}

\theoremstyle{definition}

\newtheorem{ass}[thm]{Assumption}

\theoremstyle{remark}

\begin{document}

\title[Pulse solution in a bulk--surface system]{Existence of pulse solutions in a bulk--surface reaction--diffusion system for cell polarization}

\author[R. Watanabe]{Riku Watanabe}
\address{Department of Mathematics, Faculty of Science, Hokkaido University, Hokkaido, 060-0810, Japan}
\email{watanabe.riku.y9@elms.hokudai.ac.jp}

\begin{abstract}
In this paper, we prove the existence of stationary pulse solutions for a bulk–surface reaction–diffusion system modeling cell polarization in two-dimensional domains. Using a singular perturbation method, we construct pulse solutions whose centers lie near nondegenerate critical points of a geometry-dependent potential expressed in terms of the Neumann Green's function, as formally predicted in \cite{watanabe2026}. The main analytical difficulty arises from the half--Laplacian induced by the logarithmic singularity of the Green's function, resulting in algebraically rather than exponentially decaying inner solutions.
Our analysis may also serve as a useful guide for applying singular perturbation methods to other bulk–surface type systems.
\end{abstract}

\maketitle

\section{Introduction}\label{sec:intro}
In this paper, we consider the following bulk--surface
reaction--diffusion system:
\begin{equation}\label{eq:model}
\left\{
\begin{aligned}
\eps u_t &= \eps^2 u_{ss} + f(u,v)
&&\text{on }\Gamma,\\
\eps v_t &= D\Delta v
&&\text{in }\Omega,\\
D\partial_n v &= -f(u,v)
&&\text{on }\Gamma,\\
\partial_n v &= 0
&&\text{on }\Gamma_j\quad(j=1,\cdots,N).
\end{aligned}
\right.
\end{equation}
Here, $\Omega\subset\mathbb{R}^2$ is a smooth bounded domain
whose boundary consists of $N+1$ disjoint simple closed curves,
$\partial\Omega=\Gamma\cup\Gamma_1\cup\cdots\cup\Gamma_N$,
where $N\ge0$.
The component $\Gamma$ is the reactive boundary,
whereas $\Gamma_1,\dots,\Gamma_N$ are impermeable.
The variable $s$ denotes the arc-length coordinate on $\Gamma$,
and $\partial_n$ denotes the outward normal derivative.
Here, $D>0$ and $0<\eps\ll1$, and the reaction term $f$
is of bistable type. Precise assumptions are given in Section~2.

In the standard biological interpretation, $u$ represents the concentration
of active proteins on the cell membrane and $v$ the concentration of inactive
proteins in the cytosol.
Since the diffusion of $u$ is known to be much slower than that of $v$, the diffusion coefficient is scaled using the small parameter $\eps$. The coupling condition on $\Gamma$, together with the no-flux conditions on $\Gamma_j$, gives rise to the following mass conservation law:
\begin{align}\label{eq:mass-conservation}
    M:=\int_\Gamma u(0,s)ds+\int_\Omega v(0,x)dx=\int_\Gamma u(t,s)ds+\int_\Omega v(t,x)dx\quad (t>0).
\end{align}
Mass conservation plays a key role in wave-pinning, a well-known mechanism for cell polarization \cite{mori2008}. While conventional mass-conserved reaction--diffusion models treat both components on the same domain, bulk--surface models explicitly incorporate diffusion in the cytosol and the effects of cell geometry \cite{giese2015,ramirez2015,cusseddu2019}.

We briefly review previous studies of \eqref{eq:model} and related models. For spherically symmetric domains $\Omega$, conditions for Turing instability have been investigated. These studies show that, depending on the domain geometry, a spatially homogeneous steady state can become unstable even when the two components have equal diffusion coefficients \cite{borgqvist2021,morita2020,ratz2014}. For general domains, the difficulty of rigorous analysis has motivated studies of the shadow system obtained as $D\to\infty$, as well as numerical investigations \cite{miller2023}. In this limit, $v$ becomes spatially uniform, and the system has been shown to converge to the nonlocal scalar equation
\begin{align*}
    \eps u_t=\eps^2\Delta_\Gamma u + f\bigg(u, \frac M{|\Omega|}-\frac1{|\Omega|}\int_\Gamma u(\sigma)d\sigma\bigg)
\end{align*}
known as the shadow system, which simplifies the analysis \cite{hausberg2018}. However, making $v$ spatially uniform also makes the nonlocal coupling through the bulk uniform. Numerical studies have indeed demonstrated that the dynamics of the shadow system can differ from those of the original system with finite $D$ \cite{cusseddu2022}. In particular, for the two-dimensional system considered here, the detailed effects of domain shape disappear. 

Thus, although \eqref{eq:model} incorporates the effects of cell geometry, its mathematical analysis for finite $D$ and general domains remains limited. 
In \cite{watanabe2026}, the author formally derived the ordinary differential equation \eqref{eq:ODE} describing the dynamics of pulse solutions driven by domain asymmetry. The resulting equation shows that the pulses drift according to the gradient flow of the potential function $E$ defined in \eqref{eq:E-def}. 
\eqref{eq:ODE} predicts that the critical points of $E$ correspond to stationary pulse positions. 
However, the derivation of this ordinary differential equation relies on formal arguments including a quasi-steady approximation.

This study therefore rigorously proves the existence of stationary pulse solutions near every nondegenerate critical point of the potential function $E$. Our construction is related to the singular perturbation framework developed for localized solutions of Gierer–Meinhardt systems on complex geometries. For example, spike solutions for such systems have been constructed on two-dimensional manifolds \cite{tse2010} and on Y-shaped graphs \cite{ishii2021}. A crucial difference from those systems is the rate at which the approximate profiles decay. In Gierer–Meinhardt systems, the approximate profiles decay exponentially, whereas in the present problem they decay only algebraically due to the half--Laplacian arising from the two-dimensional Neumann Green's function. Consequently, terms that are negligible in the case of exponential decay must be carefully estimated. We overcome this difficulty by establishing precise decay estimates for the approximate profiles (Proposition~\ref{prop:q_eps}). 

This paper is organized as follows. In Section~\ref{sec:main-result}, we describe the problem setting and state the main result. In Section~\ref{sec:auxiliary-problem}, we introduce the auxiliary operator and functions needed to construct approximate solutions. In Section~\ref{sec:construction}, we construct the approximate solutions and conduct error estimates. In Section~\ref{sec:reduction}, we reduce the problem to a finite-dimensional system using the Lyapunov-Schmidt method and finish the proof of the main theorem.

\section{Main result}\label{sec:main-result}
We first describe the problem setting before stating the main result. 
Let $L=|\Gamma|$ and let
$\gamma:\mathbb{R}/L\mathbb{Z}\to\Gamma$
be an arc-length parametrization.
We identify $\Gamma$ with $\mathbb{R}/L\mathbb{Z}$
through $\gamma$ and write $s\in[0,L)$
for the arc-length coordinate.
The periodic distance on $\Gamma$ is defined by
$d_\Gamma(s,\sigma)
=\min\{|s-\sigma|,L-|s-\sigma|\}
\quad (s,\sigma\in[0,L))$.
In \eqref{eq:model}, $v$ appearing in
$f(u,v)$ on $\Gamma$ is understood as
the trace $v(t,\gamma(s))$.

When $N=0$, the model describes the standard setting of a cytosol enclosed by a cell membrane. When $N\ge 1$, we also allow the cytosolic domain to contain holes. Our formulation further includes the mathematically possible case in which $u$ is defined on an inner boundary component while the outer boundary is impermeable, although we do not attach a biological interpretation to this configuration.

We assume that the nonlinearity $f$ satisfies the following conditions. 
\begin{ass}\label{ass:f}
\begin{enumerate}
\item (Bistability)
The function $f(u,v)$ is smooth on $\bR^2$. There exist constants $v_\mathrm{min}$ and $v_\mathrm{max}$ such that, for each fixed $v\in(v_\mathrm{min},v_\mathrm{max})=:I_f$, the equation $f(u,v)=0$ has exactly three roots
$h_-(v)<h_0(v)<h_+(v)$ satisfying
\[
f_{u}\!\bigl(h_\pm(v),v\bigr)<0,\qquad
f_{u}\!\bigl(h_0(v),v\bigr)>0
\]

\item (Mass Balance Condition)
Define
\[
J(v):=\int_{h_-(v)}^{h_+(v)} f(u,v)\,du \qquad (v\in I_f).
\]
There exists an isolated root $v=v_*\in I_f$ of $J(v)=0$ satisfying $J'(v_*)\neq0$. 
\item (Stability of homogeneous states)
For every $v\in I_f$, we have
\[
f_{u}\!\bigl(h_\pm(v),v\bigr)<0< f_{v}\!\bigl(h_\pm(v),v\bigr).
\]

Under this assumption, it has been shown that the homogeneous states $(u,v)=(h_\pm(v),v)$ are stable for $v\in I_f$ \cite{mori2011}. 
\end{enumerate}
\end{ass}

For example, the reaction term
\[
f(u,v)=\Bigl(k_0+\gamma_0\frac{u^2}{1+u^2}\Bigr)v-u,\quad \gamma_0>8k_0>0
\]
satisfies these assumptions. 
Differentiating $f(h_\pm(v),v)=0$ gives
\begin{align}\label{eq:fu-fv-relation}
    h_\pm'(v)f_u(h_\pm(v),v)+f_v(h_\pm(v),v)=0
\end{align}
for $v\in I_f$. 
For convenience, we define $\Delta h(v):=h_+(v)-h_-(v)$. 

We consider the stationary problem associated with \eqref{eq:model}: 
\begin{equation}\label{eq:steady}
\left\{
\begin{aligned}
0 &= \eps^2 u'' + f(u,v)\quad &&\text{on }\Gamma,\\
0 &= \Delta v \quad &&\text{in }\Omega,\\
D\partial_n v &= -f(u,v)\quad &&\text{on }\Gamma,\\
\partial_n v &= 0&&\text{on }\Gamma_j\quad(j=1,\cdots, N).
\end{aligned}
\right.
\end{equation}
Using a singular perturbation method, we construct stationary pulse solutions for which $u$ has transition layers near two points $s=s_1,s_2$ with $0<s_1<s_2<L$. In particular, we derive an equation determining the pulse center. 
We define the pulse half-width by $w:=(s_2-s_1)/2$ and the pulse center by $s_0:=(s_2+s_1)/2$ (see Figure~\ref{fig:sim}).

\begin{figure}[t]
\centering
\includegraphics[width=0.5\textwidth]{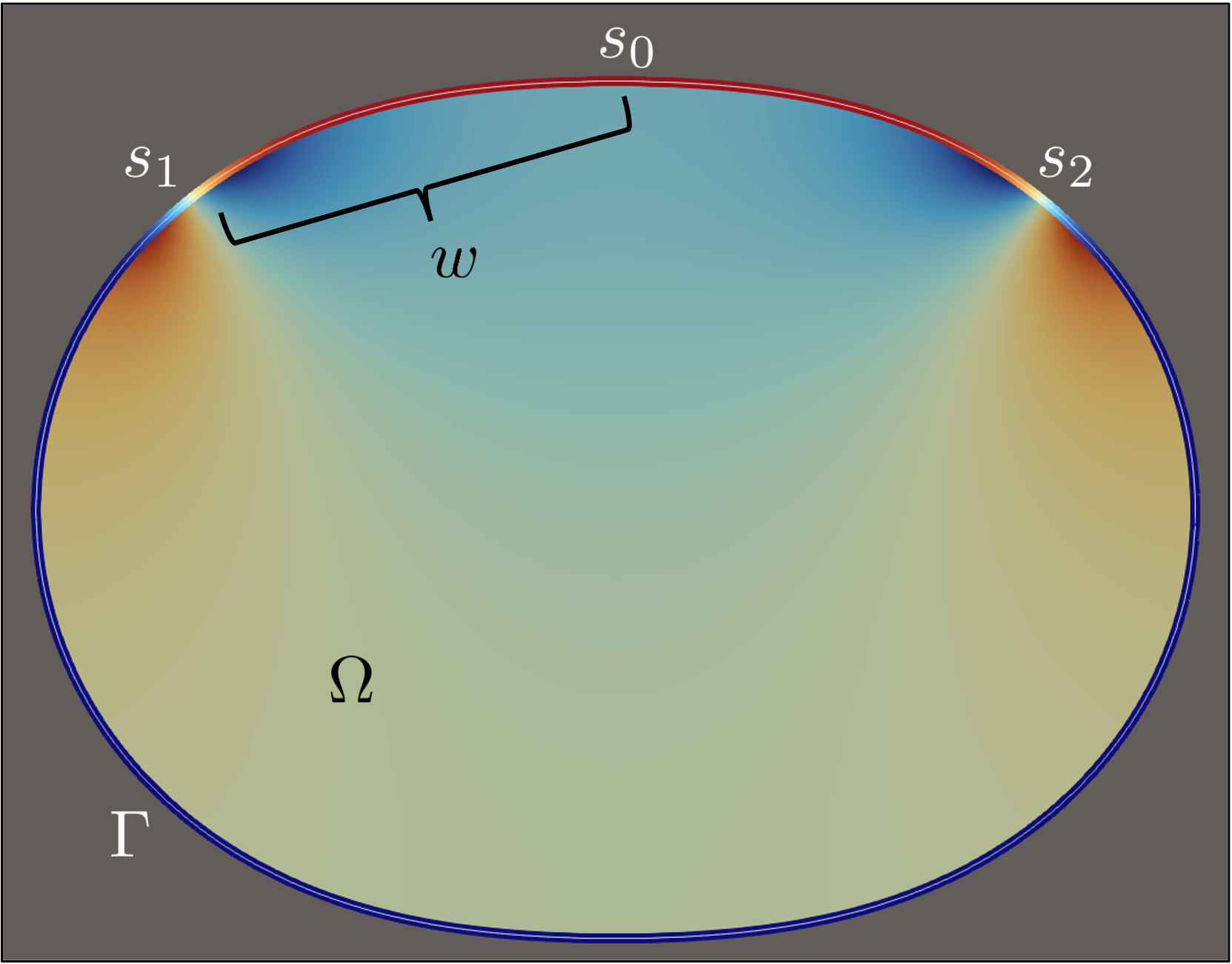}
\caption{
Schematic illustration of a stationary pulse solution. The membrane variable $u$ forms a pulse on $\Gamma$ with transition layers located near $s=s_1$ and $s=s_2$. The pulse center and half-width are denoted by $s_0=(s_1+s_2)/2$ and $w=(s_2-s_1)/2$, respectively. The bulk variable $v$ is defined in $\Omega$.
}\label{fig:sim}
\end{figure}

The Neumann Green's function is defined as the unique solution of
\begin{align*}
\begin{aligned}
\begin{cases}
\Delta_x G=\frac{1}{|\Omega|}-\delta(x-y)\quad x\in\Omega,\\
\partial_{n_x}G(x,y)=0\quad x\in\partial\Omega,\\
\int_\Omega G(x,y)dx=0    .
\end{cases}
\end{aligned}
\end{align*}
We define its restriction to $\Gamma$ by $G_\Gamma(s,\sigma):=G(\gamma(s),\gamma(\sigma))$. It admits the decomposition
\begin{align*}
    G_\Gamma(s,\sigma)=-\frac{1}{\pi}\log|\gamma(s)-\gamma(\sigma)|+H_\Gamma(s,\sigma)\quad\text{for }s,\sigma\in[0,L)
\end{align*}
into singular and regular parts. 

We assume that the total mass $M$ lies in the admissible range
\begin{align*}
    M\in\bigl(v_*|\Omega|+Lh_-(v_*),\ v_*|\Omega|+Lh_+(v_*)\bigr).
\end{align*}
Under this assumption, the leading-order pulse half-width
\begin{align*}
        w_*:=\frac{1}{2\Delta h(v_*)}\big(M-|\Omega|v_*-Lh_-(v_*)\big)
\end{align*}
satisfies $0<w_*<L/2$. 

Our main result is as follows: 
\begin{thm}\label{thm:main}
    Define the function
    \begin{align}\label{eq:E-def}
        E(s;w)=2G_\Gamma(s-w,s+w)-H_\Gamma(s+w,s+w)-H_\Gamma(s-w,s-w).
    \end{align}
    Let $s_{0,*}\in[0,L)$ be a nondegenerate critical point of $E(\cdot;w_*)$. 
    Then, for  sufficiently small $\eps$, \eqref{eq:steady} admits a solution $(u,v)=(U_\eps,V_\eps)\in C^\infty(\Gamma)\times C^\infty(\bar\Omega)$ with the following properties: 
    \begin{itemize}
    \item For any $\Delta>0$ and any $0<\rho<\min(w_*,L/2-w_*)/2$, the inequalities
        \begin{align}\label{eq:U_eps-conv}
            \begin{cases}
                |U_\eps(s)-h_+(v_*)|<\Delta\quad (d_\Gamma(s,s_{0,*})<w_*-\rho)\\
                |U_\eps(s)-h_-(v_*)|<\Delta\quad (d_\Gamma(s,s_{0,*})>w_*+\rho)
            \end{cases}
        \end{align}
    hold for  sufficiently small $\eps$. 
        \item 
        The representation
        \begin{align}\label{eq:V_eps-rep}
            V_\eps(x)=\frac{M}{|\Omega|}-\frac1{|\Omega|}\int_\Gamma U_\eps(\sigma)d\sigma+\frac{\eps^2}{D}\int_\Gamma G(x,\gamma(\sigma))U_\eps''(\sigma)d\sigma\quad (x\in\bar\Omega)
        \end{align}
    holds. Moreover, $V_\eps$ satisfies
        \begin{align}\label{eq:V_eps-conv}
            \|V_\eps-v_*\|_{L^\infty(\bar\Omega)}=O(\eps)\quad (\eps\downarrow0).
        \end{align}
    \end{itemize}
\end{thm}

In \cite{watanabe2026}, the author investigated the dynamics of pulse solutions of the original evolution problem \eqref{eq:model} and formally derived the following leading-order equation for the time-dependent pulse center $s_0(t)$: 
\begin{align}\label{eq:ODE}
    \frac{d}{dt}s_0(t)=-\eps^2\frac{\Delta h(v_*)J'(v_*)}{4D\|U'\|_{L^2(\bR)}^2}E'(s_0(t);w_*)
\end{align}
Here, $U$ is defined in \eqref{eq:U-def}. This equation describes pulse motion on the slow time scale $t=O(\eps^{-2})$. Thus, when $J'(v_*)>0$, which is the condition for stable wave-pinning, the pulse follows the gradient flow of $E(\cdot;w_*)$. 
Nevertheless, \eqref{eq:ODE} is a leading-order equation and does not itself guarantee the existence of stationary pulse solutions at critical points of $E$. Theorem~\ref{thm:main} rigorously justifies this prediction for nondegenerate critical points. Near degenerate critical points, dynamics on slower time scales may occur. 
A probabilistic interpretation of the potential function $E$ and its profiles for specific domain geometries were investigated rigorously in \cite{watanabe2026} using conformal mapping techniques, but we do not discuss these here.

\textbf{Proof strategy.}
First, we construct an approximate profile $q_\eps$ for the transition layers, which plays a central role in our analysis. A key point is to obtain precise decay estimates for $q_\eps$ and the associated singular integral $\cP q_\eps$. The introduction of this profile, together with the identification and analysis of its algebraic decay, is one of the main analytical contributions of this paper. Second, two copies of $q_\eps$ are glued to construct the approximate solution $u_\eps$, and we define $v_\eps$ using the Green's function representation. The decay estimates for $q_\eps$ and $\cP q_\eps$ are then used to obtain sufficiently sharp estimates for the residuals of these approximate solutions. Finally, we apply a Lyapunov--Schmidt reduction to construct a small correction to the approximate solution and reduce the original problem to a finite-dimensional system for the pulse width, pulse center, and the spatial average of $v$.

\section{Auxiliary problem}\label{sec:auxiliary-problem}
Define the linear operator $\cP$ by
\begin{align*}
    (\cP q)(\xi)
    &:=-\frac1\pi\int_\bR\log|\xi-\eta|q''(\eta)\,d\eta .
\end{align*}
For sufficiently regular $q$ with suitable behavior at infinity, integration by parts and a change of variables yield
\begin{align*}
    (\cP q)(\xi)
    &=-\frac1\pi \mathrm{p.v.}\int_\bR
    \frac{q'(\eta)}{\xi-\eta}\,d\eta =\frac1\pi \mathrm{p.v.}\int_\bR
    \frac{q(\xi+z)-q(\xi)}{z^2}\,dz .
\end{align*}
The operator $\cP$ is the negative half-Laplacian $-(-\partial_\xi^2)^{1/2}$.
In particular, it is a bounded linear operator $\cP:H^m(\bR)\to H^{m-1}(\bR)$
for every $m\in\bR$. 
We will use these identities throughout this section without further comment. 
This operator represents the leading-order contribution of the logarithmic singularity of the Neumann Green's function on the inner scale.
The appearance of this operator is a distinctive feature of \eqref{eq:model} and is one source of analytical difficulty. 

Define $U$ as the solution of
\begin{align}\label{eq:U-def}
    U''+f(U,v_*)=0 \quad (\xi\in\bR),\quad U(\pm\infty)=h_\pm(v_*),\quad U(0)=h_0(v_*)
\end{align}
Problem \eqref{eq:U-def} has a unique solution $U$, which is strictly increasing. Moreover, there exist constants $\gamma_U,C_m>0$ such that $\big|\partial_\xi^m(U-h_\pm(v_*))\big|\le C_me^{-\gamma_U|\xi|}$ for every $m\in\bN_0$ \cite[Lemma~2.1]{fife1973}. 

Define the linear operator $\cL:H^2(\bR)\to L^2(\bR)$ by
$\cL q
:=q''+f_u\big(U,v_*\big)q$.
It is well known that
\begin{align}\label{eq:cL-character}
\cL U'=0,\qquad
\ker\cL
=\operatorname{span}{U'}.
\end{align}

We introduce the following auxiliary problem for $q$ and $\bar v$: 
\begin{align}\label{eq:auxiliary-problem}
    \left\{
    \begin{aligned}
        &q''+f\big(q, \bar v + \frac{\eps}{D}\cP q\big)=0\quad \text{on }\bR,\\
        &q(\pm\infty)=h_\pm(\bar v),\ q(0)=U(0).
    \end{aligned}
    \right.
\end{align}
The aim of this section is to prove the following proposition. 
\begin{prop}\label{prop:q_eps}
    For  sufficiently small $\eps$, \eqref{eq:auxiliary-problem} admits a solution $(q,\bar v)=(q_\eps,\bar v_\eps)\in C^\infty(\bR)\times I_f$ which satisfies
    \begin{align}\label{eq:q_eps-convergence}
        \|q_\eps-U\|_{W^{1,\infty}(\bR)}+|\bar v_\eps-v_*|=O(\eps)\quad (\eps\downarrow0).
    \end{align}
    For every integer $j\ge0$, $q_\eps$ satisfies
    \begin{align}\label{eq:q_eps-derivative-estimate}
        \partial_\xi^j\big(q_\eps(\xi)-h_\pm(\bar v_\eps)\big)
        =O\big(\eps|\xi|^{-j-1}+e^{-\gamma_U|\xi|}\big)
    \end{align}
    as $\xi\to\pm\infty$. 
    Furthermore, for every integer $j\ge0$, as $|\xi|\to\infty$, 
    \begin{align}\label{eq:cP-q_eps-expansion}
        \partial_\xi^j\big(\cP q_\eps\big)(\xi)=O\big(|\xi|^{-j-1}\big).
    \end{align}
\end{prop}
The solution $q_\eps$ will be a key ingredient in the construction of approximate solutions. The decay estimates for $q_\eps$ and $\cP q_\eps$ established here are essential for estimating the errors of the approximate solutions. As shown in Proposition~\ref{prop:u-expansion}, $q_\eps$ provides an approximation of $u$ in the inner region, while, as shown in Proposition~\ref{prop:v-expansion1}, $\cP q_\eps$ provides an approximation of $v$ in the inner region.

The proof consists of two parts. We first construct $(q_\eps,\bar v_\eps)$ by the implicit function theorem and establish its regularity. We then estimate its algebraic decay using two properties of the singular operator $\cP$.

\subsection{Existence of $q_\eps$}
We first establish the existence of $q_\eps$. 
\begin{lem}
    For  sufficiently small $\eps$, there exists a solution $(q_\eps,\bar v_\eps)\in C^\infty(\bR)\times I_f$ of \eqref{eq:auxiliary-problem} satisfying \eqref{eq:q_eps-convergence}. Moreover, $q_\eps-h_\pm(\bar v_\eps)\in H^m(\bR_\pm)$ holds for arbitrary $m\in\bN$. 
\end{lem}
\begin{proof}
    Throughout this proof, $C$ denotes a generic constant independent of $\eps$. 
    Let $\chi_*$ be a smooth cutoff function such that $\chi_*(\xi)=0$ for $\xi\le-1$ and $\chi_*(\xi)=1$ for $\xi\ge1$. Set
    \begin{align*}
        B(\xi,\bar v):=U(\xi)+(1-\chi_*(\xi))(h_-(\bar v)-h_-(v_*))+\chi_*(\xi)(h_+(\bar v)-h_+(v_*))
    \end{align*}
    Then $B$ is smooth and satisfies $B(\pm\infty,\bar v)=h_\pm(\bar v)$. 
    We seek a solution of the form $q=B+r$ with $r\in H^2(\bR)$. The conditions $q(\pm\infty)=h_\pm(\bar v)$ are then automatically satisfied. 
    To apply the implicit function theorem, define the smooth maps $\cF:H^2(\bR)\times I_f\times\bR\to L^2(\bR)$ and $\cN:H^2(\bR)\times I_f\times\bR\to\bR$ by
    \begin{align*}
    \cF(r,\bar v,\eps)&:=(B(\cdot,\bar v)+r)''+f\big(B(\cdot,\bar v)+r, \bar v + \frac{\eps}{D}\cP (B(\cdot,\bar v)+r)\big)\\
    \cN(r,\bar v,\eps)&:=B(0,\bar v)+r(0)-U(0)
    \end{align*}
    
    Then \eqref{eq:auxiliary-problem} is equivalent to
    \begin{align}\label{eq:rewrited-auxiliary}
        \left\{
        \begin{aligned}
        \cF(r,\bar v,\eps)=0\\
        \cN(r,\bar v,\eps)=0
        \end{aligned}\right.
    \end{align}
    
    Note that $\cF(0,v_*,0)=U''+f(U,v_*)=0$ and $\cN(0,v_*,0)=B(0,v_*)-U(0)=0$. 

    We show that $D_{(r,\bar v)}(\cF,\cN)(0,v_*,0):H^2(\bR)\times\bR\to L^2(\bR)\times\bR$ has a bounded inverse. 
    Set $b(\xi):=\partial_{\bar v}B(\xi,v_*)=(1-\chi_*(\xi))h_-'(v_*)+\chi_*(\xi)h_+'(v_*)$. 
    For $(p,\nu)\in H^2(\bR)\times\bR$, we have
    \begin{align*}
        D_{(r,\bar v)}\cF(0,v_*,0)[p,\nu]&=\cL p + \nu \tilde b\\
        D_{(r,\bar v)}\cN(0,v_*,0)[p,\nu]&=b(0)\nu+p(0)
    \end{align*}
    where $\tilde b:=\cL b+f_v(U,v_*)$. For $\pm\xi\ge1$, we have $\tilde b(\xi)=f_u(U,v_*)h_\pm'(v_*)+f_v(U,v_*)$. Thus, by \eqref{eq:fu-fv-relation} and the fact that $U$ is exponentially decaying to $h_\pm(v_*)$ as $\xi\to\pm\infty$, $\tilde b\in L^2(\bR)$ holds.
    We claim that $\langle U',\tilde b\rangle=J'(v_*)\neq0$. 
    Indeed, the self-adjointness of $\cL$ and \eqref{eq:cL-character} give
    \begin{align*}
        \langle U', \tilde b\rangle
        =\int_\bR(U'\cL b+U'f_v(U,v_*))d\xi=\int_\bR b\cL U'd\xi +\int_{h_-(v_*)}^{h_+(v_*)}f_v(u,v_*)du=J'(v_*).
    \end{align*}
    We consider the equation
    \begin{align}\label{eq:inv-eq}
    \cL p =g-\nu \tilde b,\quad p(0)=\tilde\nu-b(0)\nu
    \end{align}
    for $(\tilde\nu,g)\in \bR\times L^2(\bR)$. 
    Multiplying the first equation in \eqref{eq:inv-eq} by $U'$ and integrating gives $\nu=\frac{\langle U',g\rangle}{J'(v_*)}$. The Cauchy–Schwarz inequality therefore yields $|\nu|\le C\|g\|_{L^2(\bR)}$. Since $\langle U',g-\nu\tilde b\rangle=0$, \eqref{eq:cL-character} implies that the general solution of the first equation in \eqref{eq:inv-eq} is $p=p_0+c_pU'$, where $p_0=\cL^{-1}(g-\nu\tilde b)\in H^2(\bR)$. The boundedness of $\cL^{-1}$ gives $\|p_0\|_{H^2(\bR)}\le C\|g\|_{L^2(\bR)}$. The second equation in \eqref{eq:inv-eq} determines $c_p=\frac{1}{U'(0)}(\tilde\nu-b(0)\nu-p_0(0))$. Consequently, $\|p\|_{H^2(\bR)}\le C|c_p|+C\|g-\nu\tilde b\|_{L^2(\bR)}\le C(|\tilde\nu|+\|g\|_{L^2(\bR)})$. Thus, $D_{(r,\bar v)}(\cF,\cN)(0,v_*,0)$ has a bounded inverse. The implicit function theorem now yields, for  sufficiently small $\eps$, a solution $(r_\eps,\bar v_\eps)\in H^2(\bR)\times I_f$ of \eqref{eq:rewrited-auxiliary} satisfying $\|r_\eps\|_{H^2(\bR)}+|\bar v_\eps-v_*|=O(\eps)$. The Sobolev embedding $H^2(\bR)\hookrightarrow W^{1,\infty}(\bR)$ then gives \eqref{eq:q_eps-convergence}. 
    
    Finally, the boundedness of $\cP:H^m(\bR)\to H^{m-1}(\bR)$ for every $m\in\bN$, together with a bootstrap argument and Sobolev embedding, yields the regularity of $q_\eps$. 
\end{proof}

\subsection{Algebraic decay estimates}
The preceding lemma establishes the existence of $q_\eps$. To estimate its decay rate inductively, we prepare the following two lemmas. 
\begin{lem}\label{lem:g-integral-estimate}
    Let $q\in C^\infty(\bR)\cap L^\infty(\bR)$, let $j\ge1$ be an integer, and set $g:=q^{(j)}$. Assume that $q^{(l)}(\xi)=O(|\xi|^{-l-1})$ for $1\le l<j$ as $|\xi|\to\infty$. For $j=1$, this assumption is vacuous. 
    Then there exists a constant $C$ such that the following statements hold. 
    \begin{enumerate}[label=(\roman*)]
        \item Set $r:=|\xi|/2$. Then, for large $|\xi|$, 
        \begin{align}\label{eq:g-int1}
            \bigg|\int_{|z|\ge r}\frac{g(\xi+z)}{z^2}dz\bigg|\le C(1+|\xi|)^{-j-1}
        \end{align}
        holds. 
        \item For $\xi\in\bR$, $r>0$, and $h\in C^2([\xi-r,\xi+r])$, 
        \begin{align}\label{eq:g-int2}
            \bigg|\mathrm{p.v.}\int_{|z|<r}\frac{h(\xi+z)-h(\xi)}{z^2}dz\bigg|
            \le r\sup_{|\eta-\xi|\le r}|h''(\eta)|
        \end{align}
        holds. 
    \item Assume in addition that $g(\xi)=O(|\xi|^{-j-1})$. Then, for $0<\rho<1$ and large $|\xi|$, 
        \begin{align*}
            |\cP g(\xi)|\le \frac\rho\pi\sup_{|\eta-\xi|\le\rho}|g''(\eta)|+\frac{C}{\rho}(1+|\xi|)^{-j-1}
        \end{align*}
    where $C$ is independent of $\rho$. 
    \end{enumerate}
\end{lem}
\begin{proof}
    \noindent\textit{Proof of (i). }
    Integrating by parts $j$ times gives
    \begin{align*}
        &\bigg|\int_{|z|\ge r}\frac{q^{(j)}(\xi+z)}{z^2}dz\bigg|\\
        &\le 2\sum_{l=0}^{j-2}(l+1)!r^{-l-2}\big|(-1)^{l}q^{(j-l-1)}(\xi+r)-q^{(j-l-1)}(\xi-r)\big|\\
        &\quad +2j!r^{-j-1}\big|(-1)^{j-1}q(\xi+r)-q(\xi-r)\big|+\|q\|_{L^\infty(\bR)} \int_{|z|\ge r}(j+1)!|z|^{-j-2}dz\\
        &\le C \sum_{l=0}^{j-2}r^{-l-2}\cdot r^{l-j}+Cr^{-j-1}\le C(1+|\xi|)^{-j-1},
    \end{align*}
    which proves \eqref{eq:g-int1}. 
    
    \noindent\textit{Proof of (ii). }
    \begin{align*}
        \bigg|\mathrm{p.v.}\int_{|z|<r}\frac{h(\xi+z)-h(\xi)}{z^2}dz\bigg|
        =\bigg|\int_0^r\frac{h(\xi+z)+h(\xi-z)-2h(\xi)}{z^2}dz\bigg|
        \le r\sup_{|\eta-\xi|\le r}|h''(\eta)|.
    \end{align*}
    This proves \eqref{eq:g-int2}. 
    
    \noindent\textit{Proof of (iii). }
    Set $r=|\xi|/2$ and decompose the integral as follows: 
    \begin{align*}
        \pi\cP g(\xi)&=\mathrm{p.v.}\int_{|z|<\rho}\frac{g(\xi+z)-g(\xi)}{z^2}dz
        +\int_{\rho<|z|\le r}\frac{g(\xi+z)-g(\xi)}{z^2}dz
        +\int_{|z|\ge r}\frac{g(\xi+z)}{z^2}dz-\frac{2g(\xi)}{r}
    \end{align*}
    
    By (ii), the first term is bounded by $\rho\sup_{|\eta-\xi|\le\rho}|g''(\eta)|$. 
    For the second term, the assumed decay gives
    \begin{align*}
        |(\text{2nd term})|
        \le C \sup_{\rho<|z|\le r}\big(|g(\xi+z)|+|g(\xi)|\big)\int_{|z|\ge\rho}z^{-2}dz
        \le C\rho^{-1}(1+|\xi|)^{-j-1}
    \end{align*}
    
    By (i), the third term is bounded by $C(1+|\xi|)^{-j-1}$. 
    The fourth term is bounded by $C(1+|\xi|)^{-j-2}$ by assumption. 
\end{proof}

\begin{lem}\label{lem:g-estimate}
    Let $q\in C^\infty(\bR)\cap L^\infty(\bR)$, let $j\ge1$ be an integer, and set $g:=q^{(j)}$. Assume that $q^{(l)}(\xi)=O(|\xi|^{-l-1})$ for $1\le l<j$ as $|\xi|\to\infty$. This assumption is vacuous for $j=1$. Assume also that $g(\pm\infty)=0,\ g''\in L^\infty(\bR)$. 
    Suppose that there exist positive constants $a_0,C,R_0$ such that $a,b,F\in C(\bR)$ satisfy
    \begin{align*}
        a_0\le a(\xi)\le C,\ 0\le b(\xi)\le C\  (|\xi|>R_0),\quad |F(\xi)|\le C(1+|\xi|)^{-j-1}\  (\xi\in\bR).
    \end{align*}
    
    If $g$ satisfies
    \begin{align}\label{eq:g-eq}
        -g''+ag-b\cP g=F\quad \text{on } \bR,
    \end{align}
    then
    \begin{align}\label{eq:g-ineq}
        |g(\xi)|+|g''(\xi)|+|\cP g(\xi)|\le C(1+|\xi|)^{-j-1}
    \end{align}
    holds for large $|\xi|$. 
\end{lem}
\begin{proof}
    Set $r=|\xi|/2$. 
    We rewrite \eqref{eq:g-eq} as
    \begin{align}\label{eq:g-eq-decomposed}
    \begin{aligned}
        \cA g(\xi)&:=-g''(\xi)+\bigg(a(\xi)+\frac{2b(\xi)}{\pi r}\bigg)g(\xi) 
        +\frac{b(\xi)}{\pi}\mathrm{p.v.}\int_{|z|<r}\frac{g(\xi)-g(\xi+z)}{z^2}dz\\
        &=F(\xi)+\frac{b(\xi)}{\pi}\int_{|z|\ge r}\frac{g(\xi+z)}{z^2}dz.
    \end{aligned}
    \end{align}
    By Lemma~\ref{lem:g-integral-estimate}~(i) and the assumption on $F$, the right-hand side is bounded by $C(1+|\xi|)^{-j-1}$. 
    The operator $\cA$ satisfies a maximum principle: if, for some positive constant $R\ge R_0$, 
    \begin{align*}
    p\in C(\bR)\cap C^2(\bR\setminus[-R,R]),\quad p(\xi)\le0\ (|\xi|\le R),\quad p(\pm\infty)=0,\quad \cA p(\xi)\le 0\ (|\xi|>R),
    \end{align*}
    then $p\le0$. 
    Indeed, supposing that there exists $\xi_0\in\bR\setminus[-R,R]$ such that $\max_{\xi\in\bR}p(\xi)=p(\xi_0)>0$, we then have $\cA p(\xi_0)\ge a_0p(\xi_0)>0$, thus yielding a contradiction. 
    Define $g_0(\xi):=(1+|\xi|)^{-j-1}$. By Lemma~\ref{lem:g-integral-estimate}~(ii), for sufficiently large $R$ and $|\xi|>R$, 
    \begin{align}\label{eq:cA-g0}
        \cA g_0(\xi)\ge-C(1+|\xi|)^{-j-3}+a_0 g_0(\xi)-Cr(1+|\xi|)^{-j-3}\ge\frac{a_0}{2}g_0(\xi).
    \end{align}
    For a positive constant $B$, set $p_\pm:=\pm g-Bg_0$. Taking $B$ sufficiently large, we obtain, for $|\xi|>R$, 
    \begin{align*}
        \cA p_\pm(\xi)\le C(1+|\xi|)^{-j-1}-B\frac{a_0}{2}g_0\le 0
    \end{align*}
    using \eqref{eq:g-eq-decomposed} and \eqref{eq:cA-g0}. 
    Taking $B$ sufficiently large ensures that $p_\pm\le0$ for $|\xi|\le R$. Moreover, $p_\pm(\pm\infty)=0$. The maximum principle therefore yields $|g|\le Bg_0$. 

    By \eqref{eq:g-eq} and Lemma~\ref{lem:g-integral-estimate}~(iii), for $\xi\in\bR$ and $0<\rho<1$, 
    \begin{align}\label{eq:g''-ineq}
        |g''(\xi)|\le C\rho\sup_{|\eta-\xi|\le\rho}|g''(\eta)|+C\rho^{-1}g_0.
    \end{align}
    For $\delta>0$, define
    \begin{align*}
        S_\delta:=\sup_{\xi\in\bR}\frac{|g''(\xi)|}{g_0(\xi)+\delta}
    \end{align*}
    This quantity is finite for every $\delta>0$. 
    Dividing both sides of \eqref{eq:g''-ineq} by $g_0+\delta$ and taking the supremum over $\xi\in\bR$ gives
    \begin{align*}
        S_\delta\le C\rho \sup_{\eta,\xi\in\bR,|\eta-\xi|\le\rho}\frac{g_0(\eta)+\delta}{g_0(\xi)+\delta}\frac{|g''(\eta)|}{g_0(\eta)+\delta} + C\rho^{-1}
        \le C\rho S_\delta + C\rho^{-1}
    \end{align*}
    Here, we used that $\frac{g_0(\eta)+\delta}{g_0(\xi)+\delta}\le \frac{2^{j+1}g_0(\xi)+\delta}{g_0(\xi)+\delta}\le 2^{j+1}$ for $|\xi-\eta|<1$.
    Since $C$ is independent of $\rho$, we may choose $\rho$ sufficiently small to obtain
    $S_\delta\le 2C\rho^{-1}$.
    Letting $\delta\to0$ gives $|g''|\le Cg_0$. 
    Lemma~\ref{lem:g-integral-estimate}~(iii) then yields $|\cP g|\le Cg_0$, which proves \eqref{eq:g-ineq}. 
\end{proof}

\begin{proof}[Proof of Proposition~\ref{prop:q_eps}] 
    \textbf{Step 1. Coarse algebraic decay.}
    Set
    \begin{align*}
        a=-f_u\big(q_\eps,\bar v_\eps+\frac{\eps}{D}\cP q_\eps\big),\quad
        b=\frac{\eps}{D}f_v\big(q_\eps,\bar v_\eps+\frac{\eps}{D}\cP q_\eps\big).
    \end{align*}
    These functions satisfy the assumptions of Lemma~\ref{lem:g-estimate} for sufficiently small $\eps$.
    Differentiating \eqref{eq:auxiliary-problem} once gives
    \begin{align}\label{eq:once-differentiated-auxirialy}
        -(q_\eps')''+a q_\eps' -b \cP q_\eps'=0.
    \end{align}
    Applying Lemma~\ref{lem:g-estimate} with $j=1$, we obtain
    \begin{align*}
        q_\eps'=O(|\xi|^{-2}),\quad \cP q_\eps'=O(|\xi|^{-2}).
    \end{align*}
    Integration then gives
    \begin{align*}
        q_\eps-h_\pm(\bar v_\eps)=O(|\xi|^{-1}).
    \end{align*}
    
    Let $j\ge2$ and suppose that
    \begin{align*}
        q_\eps^{(k)}=O(|\xi|^{-k-1}),\quad \cP (q_\eps^{(k)})=O(|\xi|^{-k-1})\quad(1\le k<j)
    \end{align*}
    has already been established. 
    Differentiating \eqref{eq:once-differentiated-auxirialy} $j-1$ times yields
    \begin{align*}
        -(q_\eps^{(j)})''+a q_\eps^{(j)} -b \cP q_\eps^{(j)}=F_j,\quad |F_j|=O(|\xi|^{-j-2}).
    \end{align*}
    
    Applying Lemma~\ref{lem:g-estimate} inductively, we obtain, for every $j\ge2$ as $\xi\to\pm\infty$, 
    \begin{align*}
        \partial_\xi^jq_\eps(\xi)=O(|\xi|^{-j-1}),\quad (\cP q_\eps)^{(j)}=O(|\xi|^{-j-1}),
    \end{align*}
    which proves \eqref{eq:cP-q_eps-expansion}. 

    \textbf{Step 2. Refined decay estimate.}
    We now refine these estimates to establish \eqref{eq:q_eps-derivative-estimate}. 
    Set $g_j:=\partial_\xi^j\big(q_\eps-h_\pm(\bar v_\eps)\big)$. 
    We first consider $j=0$. 
    By the mean value theorem, \eqref{eq:auxiliary-problem} can be rewritten as
    \begin{align*}
        -g_0''+a_\eps^0 g_0&=F_\eps^0
    \end{align*}
    where
    \begin{align*}
        a_\eps^0&:=-\int_0^1f_u\bigg(h_\pm(\bar v_\eps)+t(q_\eps-h_\pm(\bar v_\eps)),\bar v_\eps+\frac{\eps}{D}\cP q_\eps\bigg)dt,\\
        F_\eps^0&:=f\bigg(h_\pm(\bar v_\eps),\bar v_\eps+\frac{\eps}{D}\cP q_\eps\bigg)-f\big(h_\pm(\bar v_\eps),\bar v_\eps\big)=O(\eps|\xi|^{-1})
    \end{align*}
    
    For sufficiently large $R$, there exists a constant $a_0>0$ independent of $\eps$ such that $a_\eps^0\ge a_0$ for $|\xi|\ge R$. 
    Define $\phi_j(\xi):=e^{-\gamma_U(|\xi|-R)}+\eps|\xi|^{-j-1}$ for $|\xi|\ge R$. 
    If necessary, decrease $\gamma_U$ so that $\gamma_U^2\le a_0/2$. 
    A direct calculation then gives $-\phi_0''+a_\eps^0\phi_0\ge\frac{a_0}{2}\phi_0$ for sufficiently large $R$. Hence, for sufficiently large $C$, 
    \begin{align*}
        (-\partial_\xi^2+a_\eps^0)(\pm g_0-C\phi_0)\le0\quad (|\xi|\ge R),\quad \pm g_0(\pm R)-C\phi_0(\pm R)\le0.
    \end{align*}
    The maximum principle yields $|g_0|\le C\phi_0$, proving \eqref{eq:q_eps-derivative-estimate} for $j=0$. 

    We prove \eqref{eq:q_eps-derivative-estimate} for $j\ge1$ by induction. 
    Suppose that $|g_k|\le C\phi_k$ has been established for $0\le k<j$. 
    Differentiating \eqref{eq:auxiliary-problem} $j$ times gives
    \begin{align*}
        -g_j''+a_\eps g_j&=F_\eps^j,
    \end{align*}
    where
    \begin{align*}
        a_\eps&:=-f_u\bigg(q_\eps,\bar v_\eps+\frac{\eps}{D}\cP q_\eps\bigg),\qquad F_\eps^j:=\partial_\xi^jf\bigg(q_\eps,\bar v_\eps+\frac{\eps}{D}\cP q_\eps\bigg)+a_\eps q_\eps^{(j)}.
    \end{align*}
    The induction hypothesis and the chain rule imply that $|F_\eps^j|\le C\phi_j$. 
    Arguing as in the case $j=0$, we obtain $|g_j|\le C\phi_j$, which proves \eqref{eq:q_eps-derivative-estimate}. 
\end{proof}

\section{Construction of approximate solutions}\label{sec:construction}
We first rewrite \eqref{eq:steady} as a scalar problem. 
Writing the spatial average of $v$ as $\bar v_\eps+\eps^2\mu$, we can solve \eqref{eq:steady} for $v$ to obtain
\begin{align}\label{eq:v-solved}
    v(x)=\bar v_\eps+\eps^2\mu+\frac{\eps^2}{D}\int_\Gamma G(x,\gamma(\sigma))u''(\sigma)d\sigma
\end{align}

Define $S_{\eps,\mu}[u](s)$ by evaluating the right-hand side at $x=\gamma(s)$. Using the mass conservation law \eqref{eq:mass-conservation}, we can then rewrite \eqref{eq:steady} as the following system for $u$ and $\mu$: 
\begin{align}\label{eq:F_eps-def}
\begin{cases}
    F_{\eps,\mu}(u):=\eps^2u''+f(u,S_{\eps,\mu}[u])=0\\
    |\Omega|(\bar v_\eps+\eps^2\mu)+\int_\Gamma u(\sigma)d\sigma-M=0
\end{cases}
\end{align}

By shifting the arc-length parametrization, we may assume without loss of generality that $s_{0,*}=L/2$ in the assumptions of Theorem~\ref{thm:main}. 
Fix $0<\delta_*<\min\big(w_*,\frac12(L/2-w_*)\big)$ and $0<\delta<\frac12\min(w_*-\delta_*,L/2-w_*-2\delta_*)$. 
If $|w-w_*|<\delta_*$ and $|s_0-s_{0,*}|<\delta_*$, then $s_1=s_0-w\ge(L/2-\delta_*)-(w_*+\delta_*)>2\delta$ and $s_2\le L/2+\delta_*+w_*+\delta_*<L-2\delta$. Thus, both transition layers remain separated from $s=0,L$. Moreover, $d_\Gamma(s_1,s_2)=\min(2w,L-2w)>4\delta$, so the two transition layers are also separated from each other. 
Let $\chi$ be a smooth cutoff function satisfying
\begin{align*}
    \mathrm{supp}\chi\subset\{s|\ |s|<\delta\},\qquad 
    \chi(s)=1\quad(|s|\le\delta/2)
\end{align*}
Define the cutoff functions localized near the transition layers by $\chi_i(s):=\chi(s-s_i)$ for $i=1,2$. For $|w-w_*|<\delta_*$, we have $\mathrm{dist}_\Gamma(\mathrm{supp}\chi_1,\mathrm{supp}\chi_2)>2\delta$. 
We also define
\begin{align*}
    u_0(s;w,s_0,\bar v):=
    \begin{cases}
        h_+(\bar v)\quad s\in (s_1,s_2),\\
        h_-(\bar v)\quad s\notin (s_1,s_2)
    \end{cases}
\end{align*}

Using these functions, we define approximate solutions for $u$ and $v|_\Gamma$ by
\begin{align*}
    u_\eps(s;w,s_0)&:=u_0(s;w,s_0,\bar v_\eps)
        +\chi_1(s)\bigg\{q_\eps\bigg(\frac{s-s_1}{\eps}\bigg)-u_0(s;w,s_0,\bar v_\eps)\bigg\}\\
        &\qquad+\chi_2(s)\bigg\{q_\eps\bigg(\frac{s_2-s}{\eps}\bigg)-u_0(s;w,s_0,\bar v_\eps)\bigg\}\\
    v_\eps(s;\lambda)&=S_{\eps,\mu}[u_\eps(\cdot;w,s_0)](s)
\end{align*}
Here, we collect the parameters as $\lambda:=(\mu,w,s_0)$. Note that $u_\eps$ is smooth. Throughout the remainder of the paper, $C$ denotes a generic constant independent of $\eps$ and $\lambda$. Fix a sufficiently large constant $M_*>0$ and define $\Lambda:=\{\lambda=(\mu,w,s_0)\mid |\mu|<M_*,\ |w-w_*|<\delta_*,\ |s_0-s_{0,*}|<\delta_*\}$. 
The parameter $\lambda$ will be determined in Proposition~\ref{prop:lambda_eps-existence}. 

We summarize some properties of the Green's function which will be used hereafter.
For each $\sigma\in[0,L)$, the function $G_\Gamma(\cdot,\sigma)$ is $L$-periodic and smooth on $[0,L)\setminus\{\sigma\}$, and satisfies
\begin{align*}
G_\Gamma(s,\sigma)&=G_\Gamma(\sigma,s)\quad\text{for }s,\sigma\in[0,L),\ s\neq \sigma
\end{align*}
The function $H_\Gamma$ extends to a smooth function on all of $[0,L)\times[0,L)$ and satisfies
\begin{align*}
H_\Gamma(s,\sigma)&=H_\Gamma(\sigma,s)
\quad \text{for }s,\sigma\in[0,L)
\end{align*}
We will occasionally use the following estimate, which holds for some constant $C$ and all $s,\sigma\in[0,L)$: 
\begin{align}\label{eq:G_Gamma-est}
    |G_\Gamma(s,\sigma)|\le C(1+|\log d_\Gamma(s,\sigma)|).
\end{align}

\subsection{Expansions of approximate solutions}
The following expansions hold for $u_\eps, v_\eps$ in inner regions. 
\begin{prop}\label{prop:u-expansion}
    Uniformly for $|w-w_*|<\delta_*$, $|s_0-s_{0,*}|<\delta_*$, and $|\xi|<\delta/\eps$, we have
    \begin{align*}
        u_\eps(s_1+\eps\xi;w,s_0)=q_\eps(\xi)+O(\eps^2),\quad 
        u_\eps(s_2-\eps\xi;w,s_0)=q_\eps(\xi)+O(\eps^2)
    \end{align*}
\end{prop}
\begin{proof}
    We prove only the estimate near $s_1$. We have
    \begin{align*}
        u_\eps(s_1+\eps\xi;w,s_0)-q_\eps(\xi)=\big(1-\chi_1(s_1+\eps\xi)\big)\big(u_0(s_1+\eps\xi;w,s_0,\bar v_\eps)-q_\eps(\xi)\big)
    \end{align*}
    The right-hand side vanishes for $|\xi|<\delta/(2\eps)$. 
    For $\delta/(2\eps)\le\pm\xi\le\delta/\eps$, \eqref{eq:q_eps-derivative-estimate} gives
    \begin{align*}
        |u_\eps(s_1+\eps\xi;w,s_0)-q_\eps(\xi)|
        &\le C\big|h_\pm(\bar v_\eps)-q_\eps(\xi)\big|\le C\bigg(\frac{\eps}{|\xi|}+e^{-\gamma_U|\xi|}\bigg)\le C\eps^2
    \end{align*}
\end{proof}

\begin{prop}\label{prop:v-expansion1}
    Decompose the Green's function as $G_\Gamma(s,\sigma)=-\frac1\pi\log|s-\sigma|+\tilde G_\Gamma(s,\sigma)$, where $\tilde G_\Gamma$ is smooth. 
    Uniformly for $\lambda\in\Lambda$ and $|s-s_1|\le\delta$, we have
    \begin{align*}
        v_\eps(s;\lambda)&=\bar v_\eps+\frac{\eps}{D}(\cP q_\eps)\big((s-s_1)/\eps\big)
        +\eps^2\bigg[\mu+\frac{\Delta h(\bar v_\eps)}{D}\big(-\partial_\sigma \tilde G_\Gamma(s,s_1)+\partial_\sigma G_\Gamma(s,s_2)\big)\bigg]
        +O(\eps^3),\\
    \end{align*}
    Similarly, uniformly for $\lambda\in\Lambda$ and $|s-s_2|\le\delta$, 
    \begin{align*}
        v_\eps(s;\lambda)&=\bar v_\eps+\frac{\eps}{D}(\cP q_\eps)\big((s_2-s)/\eps\big)
        +\eps^2\bigg[\mu+\frac{\Delta h(\bar v_\eps)}{D}\big(-\partial_\sigma G_\Gamma(s,s_1)+\partial_\sigma\tilde G_\Gamma(s,s_2)\big)\bigg]
        +O(\eps^3).
    \end{align*}
\end{prop}
\begin{proof}
    See Appendix~\ref{app:v-expansion1}. 
\end{proof}

The following expansion holds for $v_\eps$ in the outer region. 
\begin{prop}\label{prop:v-expansion2}
    Uniformly for $\lambda\in\Lambda$ and $s\in[0,L)$ such that $\min\big(d_\Gamma(s,s_1),d_\Gamma(s,s_2)\big)\ge\delta/2$, we have
    \begin{align*}
        v_\eps(s;\lambda)&=\bar v_\eps + O(\eps^2)
    \end{align*}
\end{prop}
\begin{proof}
    If $\delta/2\le d_\Gamma(s,s_1)<\delta$, write $s=s_1+\eps\xi$. Then \eqref{eq:cP-q_eps-expansion} gives
    \begin{align*}
        \frac{\eps}{D}|\cP q_\eps(\xi)|\le C\eps(1+|\xi|)^{-1}\le C\eps^2.
    \end{align*}
    Together with Proposition~\ref{prop:v-expansion1}, this proves the desired estimate. The argument near $s_2$ is identical. 
    If $d_\Gamma(s,s_1)>\delta$ and $d_\Gamma(s,s_2)>\delta$, then $s$ is separated from $\mathrm{supp}\,u_\eps'$. 
    Integrating by parts in the definition of $v_\eps$ therefore yields
    \begin{align*}
        |v_\eps(s)-\bar v_\eps|&=\bigg|\eps^2\mu-\frac{\eps^2}{D}\int_\Gamma \partial_\sigma G_\Gamma(s,\sigma)u_\eps'(\sigma)d\sigma\bigg|
        \le C \eps^2(1+\|u'_\eps\|_{L^1(\Gamma)}).
    \end{align*}
    This proves the assertion. 
\end{proof}

\subsection{Residual estimates}
We now estimate the residual of the approximate solution. 
\begin{prop}\label{prop:F_eps-expansion}
    Uniformly for $\lambda\in\Lambda$ and $|\xi|<\delta/(2\eps)$, we have
    \begin{align*}
        F_{\eps,\mu}(u_\eps(\cdot;w,s_0))(s_1+\eps\xi)
        &=\eps^2\big(\mu+ R_{1,\eps}(w,s_0)\big)f_v\bigg(q_\eps(\xi), \bar v_\eps+\frac{\eps}{D}\cP q_\eps(\xi)\bigg) + O(\eps^{3}(1+|\xi|)),\\
        F_{\eps,\mu}(u_\eps(\cdot;w,s_0))(s_2-\eps\xi)
        &=\eps^2\big(\mu+ R_{2,\eps}(w,s_0)\big)f_v\bigg(q_\eps(\xi), \bar v_\eps+\frac{\eps}{D}\cP q_\eps(\xi)\bigg) + O(\eps^{3}(1+|\xi|))
    \end{align*}
    
    where
    \begin{align*}
        R_{1,\eps}(w,s_0)&:=\frac{\Delta h(\bar v_\eps)}{D}\big(-\partial_\sigma H_\Gamma(s_1,s_1)+\partial_\sigma G_\Gamma(s_1,s_2)\big),\\
        R_{2,\eps}(w,s_0)&:=\frac{\Delta h(\bar v_\eps)}{D}\big(-\partial_\sigma G_\Gamma(s_2,s_1)+\partial_\sigma H_\Gamma(s_2,s_2)\big)
    \end{align*}
\end{prop}
\begin{proof}
    Since it holds for $s,\sigma\in[0,L)$ that  
    \begin{align*}
        \tilde G_\Gamma(s,\sigma)-H_\Gamma(s,\sigma)
        =-\frac1\pi\log\frac{|\gamma(s)-\gamma(\sigma)|}{|s-\sigma|}=O(|s-\sigma|^2)\quad (s\to\sigma),
    \end{align*}
    we get 
    \begin{align}\label{eq:tildeG=H}
    \tilde G_\Gamma(s,s)=H_\Gamma(s,s).
    \end{align}
    Indeed, we can prove this as follows: since $\gamma$ is an arc-length parametrization, $|\gamma'(t)|=1$ holds and thus we get  $\gamma'(t)\cdot\gamma'(r)=1-\frac12|\gamma'(t)-\gamma'(r)|^2$ for $t,r\in[0,L)$.
    Therefore, together with the fact that $\gamma'$ is Lipschitz, we obtain
    \begin{align*}
        |\gamma(s)-\gamma(\sigma)|^2
        &=\int_\sigma^s\int_\sigma^s\gamma'(t)\cdot\gamma'(r)drdt\\
        &=|s-\sigma|^2-\frac12\int_\sigma^s\int_\sigma^s|\gamma'(t)-\gamma'(r)|^2drdt\\
        &=|s-\sigma|^2+O(|s-\sigma|^4).
    \end{align*}
    This yields
    \begin{align*}
        \log\frac{|\gamma(s)-\gamma(\sigma)|}{|s-\sigma|}
        =\log\big(1+O(|s-\sigma|^2)\big)
        =O(|s-\sigma|^2).
    \end{align*}
    
    By Taylor's theorem and \eqref{eq:tildeG=H}, we get
    \begin{align*}
        -\partial_\sigma \tilde G_\Gamma(s_1+\eps\xi,s_1)+\partial_\sigma G_\Gamma(s_1+\eps\xi,s_2)
        &=-\partial_\sigma \tilde G_\Gamma(s_1,s_1)+\partial_\sigma G_\Gamma(s_1,s_2)+O(\eps|\xi|)\\
        &=-\partial_\sigma H_\Gamma(s_1,s_1)+\partial_\sigma G_\Gamma(s_1,s_2)+O(\eps|\xi|).
    \end{align*}
    Hence, 
    \begin{align}\label{eq:partial-tilde-G-expansion}
        \frac{\Delta h(\bar v_\eps)}{D}\big(-\partial_\sigma \tilde G_\Gamma(s,s_1)+\partial_\sigma G_\Gamma(s,s_2)\big)
        =R_{1,\eps}(w,s_0)+O(\eps|\xi|).
    \end{align}
    Now, 
    \begin{align*}
        F_{\eps,\mu}(u_\eps(\cdot;w,s_0))(s)
        &=\eps^2u_\eps''(s;w,s_0)+f\big(u_\eps(s;w,s_0), v_\eps(s;\mu,w,s_0)\big).
    \end{align*}
    Since $|\xi|<\delta/(2\eps)$, 
    we have $u_\eps(s_1+\eps\xi)=q_\eps(\xi)$. Thus, Proposition~\ref{prop:v-expansion1} and \eqref{eq:partial-tilde-G-expansion} give
    \begin{align*}
        F_{\eps,\mu}(u_\eps(\cdot;w,s_0))(s_1+\eps\xi)
        &=q_\eps''(\xi)+f\bigg(q_\eps, \bar v_\eps+\frac{\eps}{D}\cP q_\eps(\xi)+\eps^2(\mu+R_{1,\eps})+O(\eps^3|\xi|)\bigg)+O(\eps^{3})\\
        &\overset{\eqref{eq:auxiliary-problem}}=\eps^2(\mu+R_{1,\eps})f_v\bigg(q_\eps(\xi), \bar v_\eps+\frac{\eps}{D}\cP q_\eps(\xi)\bigg)+O(\eps^{3}(1+|\xi|)).
    \end{align*}
    The argument near $s_2$ is identical. 
\end{proof}

\begin{prop}\label{prop:F_eps-L^2-estimate}
    Uniformly for $\lambda\in\Lambda$, we have
    \begin{align}\label{eq:F_eps-L^infty-estimate}
        \|F_{\eps,\mu}(u_\eps(\cdot;w,s_0))\|_{L^\infty(\Gamma)}=O(\eps^2)
    \end{align}
    In particular, since $\Gamma$ is compact, 
    \begin{align}\label{eq:F_eps-L^2-estimate}
        \|F_{\eps,\mu}(u_\eps(\cdot;w,s_0))\|_{L^2(\Gamma)}=O(\eps^2).
    \end{align}
\end{prop}
\begin{proof}
We divide $\Gamma$ as
\begin{align*}
[0,L)&=\{|s-s_i|<\delta/2,\ i=1,2\}\cup\{\delta/2\le|s-s_i|\le\delta,\ i=1,2\}\\
&\qquad\cup\{\min(|s-s_1|,|s-s_2|)>\delta\}=:\Gamma_1^*\cup\Gamma_2^*\cup\Gamma_3^*.
\end{align*}
Proposition~\ref{prop:F_eps-expansion} directly gives $\|F_{\eps,\mu}(u_\eps(\cdot))\|_{L^\infty(\Gamma_1^*)}=O(\eps^2)$, and Proposition~\ref{prop:u-expansion} and \ref{prop:v-expansion2} imply that $\|F_{\eps,\mu}(u_\eps(\cdot))\|_{L^\infty(\Gamma_2^*)}=O(\eps^2)$. It holds in $\Gamma_3^*$ that $u_\eps=h_\pm(\bar v_\eps)$ , and  $v_\eps=\bar v_\eps+O(\eps^2)$ by Proposition~\ref{prop:v-expansion2}. Thus, we get $\|F_{\eps,\mu}(u_\eps(\cdot))\|_{L^\infty(\Gamma_3^*)}=f(h_\pm(\bar v_\eps), \bar v_\eps)+O(\eps^2)=O(\eps^2)$. Combining these estimates yields the assertion.
\end{proof}

\section{Reduction to finite dimension}\label{sec:reduction}
We use the Lyapunov–Schmidt method to establish the existence of a true solution near the approximate solution constructed in the preceding section. 
Set the weighted Sobolev space with the norm
\begin{align*}
    \|\phi\|_{H^2_\eps(\Gamma)}:=\|\phi\|_{L^2(\Gamma)}+\eps\|\phi'\|_{L^2(\Gamma)}+\eps^2\|\phi''\|_{L^2(\Gamma)}.
\end{align*}
The Sobolev embedding estimate $\|\phi\|_{L^\infty(\Gamma)}\le C\eps^{-1/2}\|\phi\|_{H^2_\eps(\Gamma)}$ then holds. 
Define $Z_{1,\eps,w,s_0}(s):=\chi_1(s)q_\eps'((s-s_1)/\eps)$ and $Z_{2,\eps,w,s_0}(s):=\chi_2(s)q_\eps'((s_2-s)/\eps)$. Moreover, \eqref{eq:q_eps-convergence} gives, for $i=1,2$, 
\begin{align}\label{eq:Z-est}
    \|Z_{i,\eps,w,s_0}\|_{L^1(\Gamma)}=O(\eps),\quad
    \|Z_{i,\eps,w,s_0}\|_{L^2(\Gamma)}^2=\eps\|U'\|_{L^2(\bR)}^2+O(\eps^2).
\end{align}
Define the subspaces of $L^2(\Gamma)$ and $H^2_\eps(\Gamma)$ by
\begin{align*}
\cC_{\eps,w,s_0}:=\mathrm{span}\{Z_{1,\eps,w,s_0},Z_{2,\eps,w,s_0}\}\subset L^2(\Gamma),\quad
\cK_{\eps,w,s_0}:=\mathrm{span}\{Z_{1,\eps,w,s_0},Z_{2,\eps,w,s_0}\}\subset H^2_\eps(\Gamma)
\end{align*}
and 
their orthogonal complements $\cC_{\eps,w,s_0}^\perp,\cK_{\eps,w,s_0}^\perp$ with respect to the scalar product $L^2(\Gamma)$.
We then define projections by $\Pi_{\eps,w,s_0}:L^2(\Gamma)\to\cC_{\eps,w,s_0}$ and $\Pi_{\eps,w,s_0}^\perp:L^2(\Gamma)\to\cC_{\eps,w,s_0}^\perp$.  
We seek $\lambda=(\mu,w,s_0)\in\Lambda$ and $\phi\in\cK_{\eps,w,s_0}^\perp$ such that $u(s)=u_\eps(s;w,s_0)+\phi(s)$ and $\mu$ solves \eqref{eq:F_eps-def}. In particular, we require
\begin{align}\label{eq:F_eps=0}
    F_{\eps,\mu}(u_\eps(\cdot;w,s_0)+\phi)=0.
\end{align}

\subsection{Linearized operator and nonlinear term}
We decompose the left-hand side of \eqref{eq:F_eps=0} as follows: 
\begin{align*}
    F_{\eps,\mu}(u_\eps(\cdot;w,s_0)+\phi)
    =F_{\eps,\mu}(u_\eps(\cdot;w,s_0))+\cL_{\eps,\lambda}\phi+N_{\eps,\lambda}(\phi),
\end{align*}
where the operators are defined by
\begin{align*}
    \cT_\eps\phi&:=\frac{\eps^2}{D}\int_\Gamma G_\Gamma(\cdot,\sigma)\phi''(\sigma)d\sigma,\\
    \cL_{\eps,\lambda} \phi&:=\eps^2\phi''+f_u(u_\eps,v_\eps)\phi+f_v(u_\eps,v_\eps)\cT_\eps\phi,\\
    N_{\eps,\lambda}(\phi)&:=f(u_\eps+\phi,v_\eps+\cT_\eps\phi)-f(u_\eps,v_\eps)-f_u(u_\eps,v_\eps)\phi-f_v(u_\eps,v_\eps)\cT_\eps\phi.
\end{align*}

We now prove some properties of these operators.

\begin{lem}\label{lem:L-invertible}
    For  sufficiently small $\eps$ and every $\phi\in H^2_\eps(\Gamma)$, we have
    \begin{align}\label{eq:conv-estimate}
        \|\cT_\eps\phi\|_{L^2(\Gamma)}\le C\eps\|\phi\|_{H^2_\eps(\Gamma)},\quad
        \|\cT_\eps\phi\|_{L^\infty(\Gamma)}
        \le C\|\phi\|_{H^2_\eps(\Gamma)}.
    \end{align}
    
    Moreover, for  sufficiently small $\eps$, the operator
    \begin{align*}
        \Pi_{\eps,w,s_0}^\perp\cL_{\eps,\lambda}: \cK_{\eps,w,s_0}^\perp\to \cC_{\eps,w,s_0}^\perp
    \end{align*}
    is bijective, and its inverse is bounded uniformly for $\lambda\in\Lambda$ and sufficiently small $\eps$. 
\end{lem}
\begin{proof}
    We first prove \eqref{eq:conv-estimate}. 
    The derivative of the Green's function can be decomposed as
    \begin{align*}
        \partial_\sigma G_\Gamma(s,\sigma)
        &=\frac1\pi\frac{(\gamma(s)-\gamma(\sigma))\cdot\gamma'(\sigma)}{|\gamma(s)-\gamma(\sigma)|^2}+\partial_\sigma H_\Gamma(s,\sigma)
        =-\frac1L\cot\frac{\pi(\sigma-s)}{L}+\tilde G_{1,\Gamma}(s,\sigma)
    \end{align*}
    where $\tilde G_{1,\Gamma}$ is smooth. Convolution with $-\frac1L\cot\frac{\pi\sigma}{L}$ is the periodic Hilbert transform, which is bounded on $L^2$. Hence, 
    \begin{align*}
        \|\cT_\eps\phi\|_{L^2(\Gamma)}
        &=  \bigg\|-\frac{\eps^2}{D}\mathrm{p.v.}\int_\Gamma \partial_\sigma G_\Gamma(\cdot,\sigma)\phi'(\sigma)d\sigma\bigg\|_{L^2(\Gamma)}
        \le C\eps^2\|\phi'\|_{L^2(\Gamma)}\le C\eps\|\phi\|_{H^2_\eps(\Gamma)}.
    \end{align*}
    By \eqref{eq:G_Gamma-est}, we have $\sup_{s\in[0,L)}\|G_\Gamma(s,\cdot)\|_{L^2(\Gamma)}<C$. Thus, 
    \begin{align*}
        \|\cT_\eps\phi\|_{L^\infty(\Gamma)}&=\frac{\eps^2}{D}\sup_{s\in[0,L)}\big|\langle G_\Gamma(s,\cdot),\phi''\rangle\big|
        \le C\eps^2\|\phi''\|_{L^2(\Gamma)}
        \le C\|\phi\|_{H^2_\eps(\Gamma)}.
    \end{align*}
    These prove \eqref{eq:conv-estimate}. 
    
    We decompose the linearized operator as
    \begin{align*}
    \cL_{\eps,\lambda}=\cL_{\mathrm{loc},\eps}+f_v(u_\eps,v_\eps)\cT_\eps.
    \end{align*}
    
    Arguing as in \cite[Propositions~4.2 and~4.3]{wei2001}, one can show that
    $\Pi_{\eps,w,s_0}^\perp\cL_{\mathrm{loc},\eps}:\cK_{\eps,w,s_0}^\perp\to\cC_{\eps,w,s_0}^\perp$ has a bounded inverse, uniformly for $\lambda\in\Lambda$ and sufficiently small $\eps$. The desired conclusion then follows from \eqref{eq:conv-estimate} and a Neumann series argument for sufficiently small $\eps$. 
\end{proof}

The following estimates hold for $N_{\eps,\lambda}$. 
\begin{lem}\label{lem:N_eps-estimate}
    Let $\phi_1,\phi_2\in H^2_\eps(\Gamma)$ satisfy $\|\phi_i\|_{H^2_\eps(\Gamma)}\le\eps^{1/2}$ for $i=1,2$. 
    Uniformly for $\lambda\in\Lambda$ and sufficiently small $\eps$, we have
    \begin{align}\label{eq:N_eps-estimate}
        \|N_{\eps,\lambda}(\phi_1)\|_{L^2(\Gamma)}\le C\eps^{-1/2}\|\phi_1\|_{H^2_\eps(\Gamma)}^2.
    \end{align}
    Moreover, 
    \begin{align}\label{eq:N_eps-N_eps-estimate}
        \|N_{\eps,\lambda}(\phi_1)-N_{\eps,\lambda}(\phi_2)\|_{L^2(\Gamma)}\le C\eps^{-1/2}\big(\|\phi_1\|_{H^2_\eps(\Gamma)}+\|\phi_2\|_{H^2_\eps(\Gamma)}\big)\|\phi_1-\phi_2\|_{H^2_\eps(\Gamma)}.
    \end{align}
\end{lem}
\begin{proof}
    The assumptions and Sobolev embedding give $\|\phi_i\|_{L^\infty(\Gamma)}\le C$ for $i=1,2$. 
    Thus, Taylor's theorem yields
    \begin{align*}
        |N_{\eps,\lambda}(\phi_1)(s)|
        &\le C\big(|\phi_1(s)|+|\cT_\eps\phi_1(s)|\big)^2.
    \end{align*}
    Combining these estimates with \eqref{eq:conv-estimate} gives
    \begin{align*}
        \|N_{\eps,\lambda}(\phi_1)\|_{L^2(\Gamma)}
        &\le C\big(\|\phi_1\|_{L^\infty(\Gamma)}+\|\cT_\eps\phi_1\|_{L^\infty(\Gamma)}\big)\big(\|\phi_1\|_{L^2(\Gamma)}+\|\cT_\eps\phi_1\|_{L^2(\Gamma)}\big)\\
        &\le C(\eps^{-1/2}+1)\|\phi_1\|_{H^2_\eps(\Gamma)}(1+\eps)\|\phi_1\|_{H^2_\eps(\Gamma)}\\
        &\le C\eps^{-1/2}\|\phi_1\|_{H^2_\eps(\Gamma)}^2.
    \end{align*}
    This proves \eqref{eq:N_eps-estimate}. 
    Estimate \eqref{eq:N_eps-N_eps-estimate} follows similarly. 
\end{proof}

Next, we estimate the projections of the linear and nonlinear terms onto the finite--dimensional space. 

\begin{prop}\label{prop:Z-L-Z-N}
    For $\lambda\in\Lambda$ and  sufficiently small $\eps$, it holds that
    \begin{align}\label{eq:Z-L-pairing}
        |\langle Z_{i,\eps,w,s_0}, \cL_{\eps,\lambda} \phi\rangle|\le C\eps^{3/2}\|\phi\|_{H^2_\eps(\Gamma)}\quad \text{for }\phi\in \cK_{\eps,w,s_0}^\perp.
    \end{align}
    Moreover, for $\phi\in\cK_{\eps,w,s_0}^\perp$ with $\|\phi\|_{H^2_\eps(\Gamma)}\le C\eps^{1/2}$, 
    we have
    \begin{align}\label{eq:Z-N-pairing}
        |\langle Z_{i,\eps,w,s_0}, N_{\eps,\lambda}(\phi)\rangle|\le C\|\phi\|_{H^2_\eps(\Gamma)}^2.
    \end{align}
\end{prop}
\begin{proof}
    We first prove \eqref{eq:Z-L-pairing}. 
    Integration by parts gives
    \begin{align}\label{eq:Z-L-decomposition}
    \begin{aligned}
        |\langle Z_{i,\eps,w,s_0}, \cL_{\eps,\lambda} \phi\rangle|
        &\le |\langle \eps^2Z_{i,\eps,w,s_0}''+f_u(u_\eps,v_\eps)Z_{i,\eps,w,s_0},\phi\rangle|
        +\big|\big\langle Z_{i,\eps,w,s_0},f_v(u_\eps,v_\eps)\cT_\eps\phi\big\rangle\big|\\
        &\le \|\eps^2Z_{i,\eps,w,s_0}''+f_u(u_\eps,v_\eps)Z_{i,\eps,w,s_0}\|_{L^2(\Gamma)}\|\phi\|_{L^2(\Gamma)}\\
        &\qquad  +\|Z_{i,\eps,w,s_0}\|_{L^2(\Gamma)}\big\|f_v(u_\eps,v_\eps)\cT_\eps\phi\big\|_{L^2(\Gamma)}.
    \end{aligned}
    \end{align}
    
    We estimate the first term. 
    Differentiating \eqref{eq:auxiliary-problem} with respect to $\xi$ yields
    \begin{align*}
        q_\eps'''+f_u\big(q_\eps,\bar v_\eps+\frac{\eps}{D}\cP q_\eps\big)q_\eps'=
        -f_v\big(q_\eps,\bar v_\eps+\frac{\eps}{D}\cP q_\eps\big)\frac{\eps}{D}(\cP q_\eps)'.
    \end{align*}
    
    Using this identity, we obtain for $i=1$, for $\sigma=s_1+\eps\eta\in[s_1-\delta,s_1+\delta]$, 
    \begin{align*}
        &\eps^2Z_{1,\eps,w,s_0}(\sigma)''+f_u(u_\eps(\sigma),v_\eps(\sigma))Z_{1,\eps,w,s_0}(\sigma)\\
        &=\chi_1(\sigma)\bigg(q_\eps'''(\eta)+f_u\big(q_\eps(\eta),\bar v_\eps+\frac{\eps}{D}\cP q_\eps(\eta)\big)q_\eps'(\eta)\bigg)
        \\
        &\quad +\chi_1(\sigma)\bigg(f_u\big(u_\eps(\sigma),v_\eps(\sigma)\big)-f_u\big(q_\eps(\eta),\bar v_\eps
        +\frac{\eps}{D}\cP q_\eps(\eta)\big)\bigg)q_\eps'(\eta)
        +2\eps\chi_1'(\sigma)q_\eps''(\eta)+\eps^2\chi_1''(\sigma)q_\eps'(\eta)\\
        &\overset{Prop.~\ref{prop:u-expansion},\ref{prop:v-expansion1}}=
        -\chi_1(\sigma)f_v\big(q_\eps(\eta),\bar v_\eps+\frac{\eps}{D}\cP q_\eps(\eta)\big)\frac{\eps}{D}(\cP q_\eps)'(\eta) +O(\eps^2)q_\eps'(\eta)
        +2\eps\chi_1'(\sigma)q_\eps''(\eta)+\eps^2\chi_1''(\sigma)q_\eps'(\eta)
    \end{align*}
    Together with $d\sigma=\eps\,d\eta$, this gives
    \begin{align}\label{eq:Z-L-1st-term}
    \begin{aligned}
        &\|\eps^2Z_{1,\eps,w,s_0}''+f_u(u_\eps,v_\eps)Z_{1,\eps,w,s_0}\|_{L^2(\Gamma)}\\
        &\le C\eps^{3/2}\|(\cP q_\eps)'\|_{L^2(\bR)}+C\eps^{5/2}\|q_\eps'\|_{L^2(\bR)}+C\eps^{3/2}\|q_\eps''\|_{L^2(\bR)}
        +C\eps^{5/2}\|q_\eps'\|_{L^2(\bR)}=O(\eps^{3/2}).
    \end{aligned}
    \end{align}
    We next estimate the second term in \eqref{eq:Z-L-decomposition}. By \eqref{eq:Z-est} and \eqref{eq:conv-estimate}, 
    \begin{align}\label{eq:Z-L-2nd-term}
        \|Z_{i,\eps,w,s_0}\|_{L^2(\Gamma)}\big\|f_v(u_\eps,v_\eps)\cT_\eps\phi\big\|_{L^2(\Gamma)}
        &\le C\eps^{1/2}\cdot\eps\|\phi\|_{H^2_\eps(\Gamma)}.
    \end{align}
    Combining \eqref{eq:Z-L-1st-term} and \eqref{eq:Z-L-2nd-term} yields \eqref{eq:Z-L-pairing} for $i=1$. The case $i=2$ is analogous. 

    Estimate \eqref{eq:Z-N-pairing} follows from
    Lemma~\ref{lem:N_eps-estimate} and \eqref{eq:Z-est} by the Cauchy–Schwarz inequality. 
    
\end{proof}

We define
\begin{align*}
        R_{1,*}(w,s_0)&:=\frac{\Delta h(v_*)}{D}\big(-\partial_\sigma H_\Gamma(s_1,s_1)+\partial_\sigma G_\Gamma(s_1,s_2)\big),\\
        R_{2,*}(w,s_0)&:=\frac{\Delta h(v_*)}{D}\big(-\partial_\sigma G_\Gamma(s_2,s_1)+\partial_\sigma H_\Gamma(s_2,s_2)\big).
\end{align*}

\begin{lem}\label{lem:Z-F-estimate}
    For  sufficiently small $\eps$, uniformly for $\lambda\in\Lambda$, we have
    \begin{align*}
        \langle Z_{i,\eps,w,s_0},F_{\eps,\mu}(u_\eps(\cdot;w,s_0))\rangle=\eps^3(\mu+R_{i,*})J'(v_*)+O(\eps^{4})
    \end{align*}
\end{lem}
\begin{proof}
    We prove the case $i=1$; the case $i=2$ is analogous.
    \begin{align*}
        &\langle Z_{1,\eps,w,s_0},F_{\eps,\mu}(u_\eps(\cdot;w,s_0))\rangle\\
        &=\int_\Gamma\chi_1(\sigma)q_\eps'\big((\sigma-s_1)/\eps\big)F_{\eps,\mu}(u_\eps)(\sigma)d\sigma\\
        &=\eps\int_{|\eta|<\delta/(2\eps)}q_\eps'(\eta)F_{\eps,\mu}(u_\eps)(s_1+\eps\eta)d\eta
        +\eps\int_{\delta/(2\eps)<|\eta|<\delta/\eps}\chi(\eps\eta)q_\eps'(\eta)F_{\eps,\mu}(u_\eps)(s_1+\eps\eta)d\eta
    \end{align*}
    \begin{align*}
        \text{(1st term)}&\overset{Prop.~\ref{prop:F_eps-expansion}}
        =\eps^3(\mu+R_{1,\eps})\int_{|\eta|<\delta/(2\eps)} q_\eps'f_v\big(q_\eps,\bar v_\eps+\frac\eps D\cP q_\eps\big)d\eta\\
        &\qquad+\int_{|\eta|<\delta/(2\eps)}q_\eps'(\eta)O(\eps^{4}(1+|\eta|))d\eta\\
        &=\eps^3(\mu+R_{1,\eps})\int_\bR q_\eps'f_v(q_\eps,\bar v_\eps)d\eta+O(\eps^{4})\\
        &=\eps^3(\mu+R_{1,*})J'(\bar v_\eps)+O(\eps^{4}),\\
        &=\eps^3(\mu+R_{1,*})J'(v_*)+O(\eps^{4}),\\
        |\text{(2nd term)}|&\le \eps\|F_{\eps,\mu}(u_\eps)\|_{L^\infty(\Gamma)}\int_{\delta/(2\eps)<|\eta|<\delta/\eps}|q_\eps'(\eta)|d\eta\\
        &\overset{\eqref{eq:F_eps-L^infty-estimate}}\le C\eps^3\int_{\delta/(2\eps)<|\eta|<\delta/\eps}|q_\eps'(\eta)|d\eta=O(\eps^{5}).
    \end{align*}
    Consequently, 
    \begin{align*}
        \langle Z_{i,\eps,w,s_0},F_{\eps,\mu}(u_\eps(\cdot;w,s_0))\rangle=\eps^3(\mu+R_{i,*})J'(v_*)+O(\eps^{4}).
    \end{align*}
    
\end{proof}

\subsection{Finite-dimensional problem}
Applying the projections, we decompose \eqref{eq:F_eps=0} as
\begin{align}\label{eq:F_eps=0-divided}
    \left\{
    \begin{aligned}
        \Pi_{\eps,w,s_0}^\perp F_{\eps,\mu}(u_\eps(\cdot;w,s_0)+\phi)&=0,\\
        \Pi_{\eps,w,s_0} F_{\eps,\mu}(u_\eps(\cdot;w,s_0)+\phi)&=0.
    \end{aligned}
    \right.
\end{align}

The following proposition establishes the existence of $\phi$ satisfying the first equation of \eqref{eq:F_eps=0-divided}.

\begin{prop}\label{prop:phi_eps-existence}
    For every $\lambda\in\Lambda$ and  sufficiently small $\eps$, 
    there exists $\phi=\phi_{\eps,\lambda}\in\cK_{\eps,w,s_0}^\perp$ satisfying the first equation in \eqref{eq:F_eps=0-divided} and
    \begin{align*}
        \|\phi_{\eps,\lambda}\|_{H^2_\eps(\Gamma)}=O(\eps^2).
    \end{align*}
    Moreover, $\phi_{\eps,\lambda}$ depends continuously on $\lambda$. 
\end{prop}
\begin{proof}
    The proof is standard. Rewrite the first equation in \eqref{eq:F_eps=0-divided} as
    \begin{align*}
        \phi=-(\Pi_{\eps,w,s_0}^\perp\cL_{\eps,\lambda})^{-1}\Pi^\perp_{\eps,w,s_0}(F_{\eps,\mu}(u_\eps(\cdot;w,s_0))+N_{\eps,\lambda}(\phi))=:T_\eps(\phi).
    \end{align*}
    This reduces the problem to finding a fixed point of $T_\eps$. By Lemma~\ref{lem:N_eps-estimate}, for a sufficiently large constant $R>0$ and sufficiently small $\eps$, the map $T_\eps$ sends the ball of radius $R\eps^2$ in $\cK_{\eps,w,s_0}^\perp$ into itself. Moreover, \eqref{eq:F_eps-L^2-estimate} and Lemma~\ref{lem:N_eps-estimate} imply that this map is a contraction for sufficiently small $\eps$. The contraction mapping principle yields the desired conclusion. Proving the continuous dependence of $\phi_{\eps,\lambda}$ on $\lambda$ is straightforward.
\end{proof}

Thus the infinite-dimensional part of the problem has been solved; it remains to determine the three parameters $\mu,w,s_0$ that satisfy the second equation in \eqref{eq:F_eps=0-divided} and the mass conservation law \eqref{eq:mass-conservation}. 
For $i=1,2$, define
\begin{align*}
\begin{aligned}
    g_{i,\eps}(\lambda)&:=\langle Z_{i,\eps,w,s_0},F_{\eps,\mu}(u_\eps(\cdot;w,s_0)+\phi_{\eps,\lambda})\rangle,\\
    G_\eps(\lambda)
    &:=
    \begin{pmatrix}
    |\Omega|(\bar v_\eps+\eps^2\mu)+\int_\Gamma\big(u_\eps(\sigma;w,s_0)+\phi_{\eps,\lambda}(\sigma)\big)d\sigma-M\\
    \eps^{-3}g_{1,\eps}(\lambda)\\
    \eps^{-3}g_{2,\eps}(\lambda)
    \end{pmatrix}.
\end{aligned}
\end{align*}
We then seek a solution $\lambda\in\Lambda$ of
$G_\eps(\lambda)=0$.
We also define
\begin{align*}
    \quad g_{i,*}(\lambda):=J'(v_*)\big(\mu+R_{i,*}(w,s_0)\big),
    \quad
    G_*(\lambda)
    &:=
    \begin{pmatrix}
    2\Delta h(v_*)(w-w_*)\\
    g_{1,*}(\lambda)\\
    g_{2,*}(\lambda)
    \end{pmatrix}.
\end{align*}
The following uniform convergence holds. 

\begin{lem}\label{lem:G_eps-G_*-convergence}
    For  sufficiently small $\eps$, 
    \begin{align*}
        \sup_{\lambda\in \Lambda}|G_\eps(\lambda)-G_*(\lambda)|=O(\eps^{1/2}).
    \end{align*}
\end{lem}
\begin{proof}
    We estimate each component as follows: 
    \begin{align*}
        &|\Omega|(\bar v_\eps+\eps^2\mu)+\int_\Gamma\big(u_\eps(\sigma;w,s_0)+\phi_{\eps,\lambda}(\sigma)\big)d\sigma-M\\
        &=
        |\Omega|v_*+\int_\Gamma u_0(\sigma;w,s_0,v_*)d\sigma-M+O(\eps)\\
        &=2\Delta h(v_*)(w-w_*)+O(\eps),
    \end{align*}
    \begin{align*}
        &|\eps^{-3}g_{i,\eps}(\lambda)-g_{i,*}(\lambda)|\\
        &\le \big|\eps^{-3}\langle Z_{i,\eps,w,s_0},F_{\eps,\mu}(u_\eps(\cdot;w,s_0))\rangle-g_{i,*}(\lambda)\big|
        +\eps^{-3}\big|\langle Z_{i,\eps,w,s_0},\cL_{\eps,\lambda} \phi_{\eps,\lambda}+N_{\eps,\lambda}(\phi_{\eps,\lambda})\rangle\big|\\
        &\le C\eps+C\eps^{-3}\big(\eps^{3/2}\|\phi_{\eps,\lambda}\|_{H^2_\eps(\Gamma)}+\|\phi_{\eps,\lambda}\|_{H^2_\eps(\Gamma)}^2\big)\\
        &\le C\eps^{1/2}.
    \end{align*}
\end{proof}

\begin{prop}\label{prop:lambda_eps-existence}
    Let $s_{0,*}=L/2$ be a nondegenerate critical point of $E(\cdot;w_*)$. 
    Set 
    \begin{align*}
    \mu_*:=-\frac12(R_{1,*}(w_*,s_{0,*})+R_{2,*}(w_*,s_{0,*})),\quad \lambda_*:=(\mu_*,w_*,s_{0,*}).
    \end{align*}
    Then, for  sufficiently small $\eps$, 
    there exists $\lambda=\lambda_\eps:=(\mu_\eps,w_\eps,s_{0,\eps})$ satisfying $G_\eps(\lambda)=0$ and
    $\lambda_\eps\to\lambda_*$ as $\eps\downarrow0$. 
\end{prop}
\begin{proof}
    Since $\frac{\Delta h(v_*)}{2D}E'(s_0;w)=R_{1,*}(w,s_0)-R_{2,*}(w,s_0)$, we have $G_*(\lambda_*)=0$. 
    Moreover, 
    \begin{align*}
        \det D_\lambda G_*(\lambda_*)=\frac{(\Delta h(v_*)J'(v_*))^2}{D}E''(s_{0,*};w_*)\neq0.
    \end{align*}
    
    Thus, we can choose $r_*>0$ sufficiently small so that the open ball $B_{r_*}(\lambda_*)$ is contained in $\Lambda$, $\min_{\partial B_{r_*}(\lambda_*)}|G_*|>0$, and $\lambda_*$ is the only zero of $G_*$ in $B_{r_*}(\lambda_*)$. By Lemma~\ref{lem:G_eps-G_*-convergence}, the homotopy $tG_\eps+(1-t)G_*\quad(0\le t\le1)$ satisfies, on $\partial B_{r_*}(\lambda_*)$, 
    \begin{align*}
        |tG_\eps+(1-t)G_*|\ge |G_*|-t|G_*-G_\eps|\ge \frac12\min_{\partial B_{r_*}(\lambda_*)}|G_*|>0
    \end{align*}
    for  sufficiently small $\eps$. The proposition now follows from standard degree theory. 
\end{proof}

\begin{proof}[Completion of the proof of Theorem~\ref{thm:main}]
    Propositions~\ref{prop:phi_eps-existence} and~\ref{prop:lambda_eps-existence} imply that \eqref{eq:F_eps-def} admits a solution $U_\eps(s)=u_\eps(s;w_\eps,s_{0,\eps})+\phi_{\eps,\lambda_\eps}\in H^2(\Gamma)$ satisfying \eqref{eq:U_eps-conv}. Equation \eqref{eq:v-solved} and the second equation in \eqref{eq:F_eps-def} yield
    \eqref{eq:V_eps-rep}. By \eqref{eq:q_eps-convergence} and \eqref{eq:conv-estimate}, we obtain $\|V_\eps|_\Gamma-v_*\|_{L^\infty(\Gamma)}\le\|v_\eps-v_*\|_{L^\infty(\Gamma)}+\|\cT_\eps\phi_{\eps,\lambda_\eps}\|_{L^\infty(\Gamma)}=O(\eps)$ as $\eps\downarrow0$. 
    Since $V_\eps$ is harmonic in $\Omega$ and $\partial_n V_\eps=0$ on $\Gamma_j\ (j=1,\cdots, N)$, the maximum principle and Hopf's lemma give $\|V_\eps-v_*\|_{L^\infty(\bar\Omega)}=\|V_\eps|_\Gamma-v_*\|_{L^\infty(\Gamma)}$, which yields \eqref{eq:V_eps-conv}. Moreover, $V_\eps\in C^\infty(\Omega)$. Since $D\partial_nV_\eps=\eps^2U_\eps''\in L^2(\Gamma)$, by the standard regularity theory for the Neumann problem, we have $V_\eps|_\Gamma\in H^1(\Gamma)$. Thus we get $D\partial_n V_\eps=\eps^2U_\eps''=-f(U_\eps,V_\eps|_\Gamma)\in H^1(\Gamma)$, and iterating this argument, we get $V_\eps|_\Gamma\in H^k(\Gamma),\ U_\eps''\in H^k(\Gamma)$ for all $k\in\bN$ and thus $V_\eps|_\Gamma,U_\eps\in C^\infty(\Gamma)$ by Sobolev embedding. Then Schauder estimates yield $V_\eps\in C^\infty(\bar \Omega)$. 
\end{proof}

\section{Conclusion}
In this paper, we have rigorously proved the existence of stationary pulse solutions of \eqref{eq:model} near every nondegenerate critical point of the potential function $E$. 
In this model, singular integrals arising from the Green's function enter the construction of approximate solutions and lead to algebraic decay of the inner approximate profile $q_\eps$. 
Such singularities are expected to arise more generally in two-dimensional bulk–surface problems. The main contribution of this paper is to provide a rigorous treatment of these singularities and thereby suggest an approach to other problems with dynamic boundary conditions. 
The stability of the solutions constructed here has not been investigated in this paper. The ODEs derived formally by the author in \cite{watanabe2026} suggest that the stability of the pulse is determined by the signs of $E''(s_{0,*};w_*)$ and $J'(v_*)$. A rigorous proof of this prediction is an interesting question for future work.

\par\vspace{\baselineskip}
\noindent\textbf{Acknowledgements}\\
This work was supported by JST SPRING, Grant Number JPMJSP2119. 

\par\vspace{\baselineskip}
\noindent\textbf{Declaration of Competing Interest}\\
The author declares that they have no known competing financial interests or personal relationships that could have appeared to influence the work reported in this paper. 

\par\vspace{\baselineskip}
\noindent\textbf{Use of generative AI}\\
OpenAI's ChatGPT was used for English-language editing and as a discussion tool during the preparation of this manuscript. All mathematical statements, derivations, references, and conclusions were independently verified by the author.

\appendix
\section{Proof of Proposition~\ref{prop:v-expansion1}}\label{app:v-expansion1}
We prove only the case $s\sim s_1$. Write $s=s_1+\eps\xi$, so that $|\xi|\le\delta/\eps$ and $|s-s_1|\le\delta$. We decompose
\begin{align*}
    v_\eps(s;\lambda)&=\bar v_\eps+\eps^2\mu
    -\frac{\eps^2}{\pi D}\int_{|\sigma-s_1|<\delta/2}\log|s-\sigma|u_\eps''(\sigma)d\sigma
    +\frac{\eps^2}{D}\int_{|\sigma-s_1|<\delta/2}\tilde G_\Gamma(s,\sigma)u_\eps''(\sigma)d\sigma\\
    &\quad +\frac{\eps^2}{D}\int_{|\sigma-s_1|\ge\delta/2,\ |\sigma-s_2|\ge\delta}G_\Gamma(s,\sigma)u_\eps''(\sigma)d\sigma
    +\frac{\eps^2}{D}\int_{|\sigma-s_2|<\delta}G_\Gamma(s,\sigma)u_\eps''(\sigma)d\sigma\\
    &=:\bar v_\eps+\eps^2\mu+I_1(s)+I_2(s)+I_3(s)+I_4(s)
\end{align*}
and estimate each term separately. 

\textbf{Step 1: Estimate of $I_1$}\\
Set $\sigma=s_1+\eps\eta$. On this integration region, we have $u_\eps(\sigma)=q_\eps(\eta)$. 
The change of variables $t=\eps\eta$ gives
\begin{align*}
    &\bigg|I_1(s_1+\eps\xi)-\frac\eps D\cP q_\eps(\xi)\bigg|\\
    &=\bigg|\frac{\eps}{\pi D}\int_{|\eta|\ge\delta/(2\eps)}\log|\eps(\xi-\eta)|q_\eps''(\eta)d\eta\bigg|\le C\int_{|t|\ge\delta/2}\big|\log|\eps\xi-t|\big|\eps^4|t|^{-3}dt\le C\eps^4
\end{align*}
Consequently, 
\begin{align}\label{eq:STEP1-2}
    I_1(s)=\frac\eps D\cP q_\eps(\xi)+O(\eps^4)
\end{align}

\textbf{Step 2: Estimate of $I_2$}\\
Set $\sigma=s_1+\eps\eta$. On this integration region, we have $u_\eps(\sigma)=q_\eps(\eta)$. 
Integration by parts gives
\begin{align*}
    I_2(s)&=\frac\eps D\bigg[\tilde G_\Gamma(s,s_1+\eps\eta)q_\eps'(\eta)\bigg]_{-\delta/(2\eps)}^{\delta/(2\eps)}
    -\frac{\eps^2} D\int_{|\eta|<\delta/(2\eps)}\partial_\sigma\tilde G_\Gamma(s,s_1+\eps\eta)q_\eps'(\eta)d\eta\\
    &= O(\eps^4)-\frac{\eps^2} D\int_{|\eta|<\delta/(2\eps)}\partial_\sigma\tilde G_\Gamma(s,s_1+\eps\eta)q_\eps'(\eta)d\eta.
\end{align*}
By Taylor's theorem, 
\begin{align*}
    &\bigg|\int_{|\eta|<\delta/(2\eps)}\big(\partial_\sigma\tilde G_\Gamma(s,s_1+\eps\eta)-\partial_\sigma\tilde G_\Gamma(s,s_1)\big)q_\eps'(\eta)d\eta\bigg|\\
    &\le C\eps\int_{|\eta|<\delta/(2\eps)}|\eta q_\eps'(\eta)|d\eta\le C\eps.
\end{align*}
Combining these estimates with $\int_{|\eta|<\delta/(2\eps)}q_\eps'(\eta)\,d\eta=\Delta h(\bar v_\eps)+O(\eps^2)$ yields
\begin{align}\label{eq:STEP2-2}
    I_2(s)=-\frac{\eps^2\Delta h(\bar v_\eps)}{D}\partial_\sigma\tilde G_\Gamma(s,s_1)+O(\eps^3)
\end{align}

\textbf{Step 3: Estimate of $I_3$}\\
On the integration region for $I_3$, where $|\sigma-s_1|\ge\delta/2$ and $|\sigma-s_2|\ge\delta$, we have $\chi_2(\sigma)=0$. Moreover, when $\delta/2\le|\sigma-s_1|<\delta$, 
\begin{align*}
    |u_\eps''(\sigma)|&\le C\big\{\big|q_\eps\big((\sigma-s_1)/\eps\big)-u_0(\sigma)\big|+\eps^{-1}\big|q_\eps'\big((\sigma-s_1)/\eps\big)\big|+\eps^{-2}\big|q_\eps''\big((\sigma-s_1)/\eps\big)\big|\big\}\\
    &\le C\big\{\eps\cdot2\eps/\delta+\eps^{-1}\cdot\eps(2\eps/\delta)^2+\eps^{-2}\cdot\eps(2\eps/\delta)^3\big\}\\
    &=O(\eps^2)
\end{align*}
holds. In the remaining region, $u_\eps''(\sigma)=0$ holds.
Therefore, \eqref{eq:G_Gamma-est} implies
\begin{align*}
    |I_3(s)|\le C\eps^4\int_\Gamma|G_\Gamma(s,\sigma)|d\sigma\le C\eps^4\int_\Gamma\big(1+|\log d_\Gamma(s,\sigma)|\big)d\sigma=O(\eps^4).
\end{align*}
Thus, 
\begin{align}\label{eq:STEP3-2}
    I_3(s)=O(\eps^4)
\end{align}

\textbf{Step 4: Estimate of $I_4$}\\
On the integration region for $I_4$, we have $|\sigma-s_2|<\delta$. Since $|s-s_1|\le\delta$, the points $s$ and $\sigma$ remain uniformly separated. Consequently, $\partial_\sigma^jG_\Gamma(s,\sigma)$ for $j=0,1,2$ are uniformly bounded on this region. 
Moreover, $u_\eps'=O(\eps^2)$ for $\delta/2<|\sigma-s_2|<\delta$, and
\begin{align*}
    I_4(s)&=\frac{\eps^2}{D}\bigg[G_\Gamma(s,\sigma)u_\eps'(\sigma)\bigg]_{s_2-\delta}^{s_2+\delta}
    -\frac{\eps^2}{D}\int_{|\sigma-s_2|<\delta}\partial_\sigma G_\Gamma(s,\sigma)u_\eps'(\sigma)d\sigma\\
    &=0-\frac{\eps^2}{D}\int_{|\sigma-s_2|<\delta}\partial_\sigma G_\Gamma(s,\sigma)u_\eps'(\sigma)d\sigma.
\end{align*}
Using $\int_{|\sigma-s_2|<\delta}u_\eps'(\sigma)\,d\sigma=-\Delta h(\bar v_\eps)$, we obtain
\begin{align*}
    \bigg|I_4(s)-\frac{\eps^2\Delta h(\bar v_\eps)}{D}\partial_\sigma G_\Gamma(s,s_2)\bigg|
    &\le\frac{\eps^2}{D}\int_{|\sigma-s_2|<\delta}\big|\partial_\sigma G_\Gamma(s,\sigma)-\partial_\sigma G_\Gamma(s,s_2)\big||u_\eps'(\sigma)|d\sigma\\
    &\le C\eps^2\int_{|\sigma-s_2|<\delta}|\sigma-s_2|\big|u_\eps'(\sigma)\big|d\sigma\\
    &\le C\eps^2\bigg(\frac1\eps\int_{|\sigma-s_2|<\delta/2}|\sigma-s_2|\big|q_\eps'\big((s_2-\sigma)/\eps\big)\big|d\sigma+C\eps^2\bigg)\\
    &\overset{\sigma=s_2+\eps\eta}\le C\eps^3\int_{|\eta|<\delta/(2\eps)}|\eta q_\eps'(\eta)|d\eta +C\eps^4 \le C\eps^3.
\end{align*}
Hence, 
\begin{align}\label{eq:STEP4-2}
    I_4(s)=\frac{\eps^2\Delta h(\bar v_\eps)}{D}\partial_\sigma G_\Gamma(s,s_2)+O(\eps^3)
\end{align}

Combining \eqref{eq:STEP1-2}--\eqref{eq:STEP4-2} proves the desired result.

\bibliographystyle{unsrtnat}
\bibliography{reference}

@article{tse2010,
  title = {The {{Gierer}}--{{Meinhardt}} System on a Compact Two-Dimensional {{Riemannian}} Manifold: {{Interaction}} of {{Gaussian}} Curvature and {{Green}}'s Function},
  shorttitle = {The {{Gierer}}--{{Meinhardt}} System on a Compact Two-Dimensional {{Riemannian}} Manifold},
  author = {Tse, Wang Hung and Wei, Juncheng and Winter, Matthias},
  year = {2010},
  month = oct,
  journal = {Journal de Math\'ematiques Pures et Appliqu\'ees},
  volume = {94},
  number = {4},
  pages = {366--397},
  doi = {10.1016/j.matpur.2010.03.003},
  langid = {english},
}

@article{mori2008,
  title = {Wave-{{Pinning}} and {{Cell Polarity}} from a {{Bistable Reaction-Diffusion System}}},
  author = {Mori, Yoichiro and Jilkine, Alexandra and {Edelstein-Keshet}, Leah},
  year = {2008},
  month = may,
  journal = {Biophysical Journal},
  volume = {94},
  number = {9},
  pages = {3684--3697},
  doi = {10.1529/biophysj.107.120824},
}

@article{cusseddu2019,
  title = {A Coupled Bulk-Surface Model for Cell Polarisation},
  author = {Cusseddu, D. and {Edelstein-Keshet}, L. and Mackenzie, J.A. and Portet, S. and Madzvamuse, A.},
  year = {2019},
  month = nov,
  journal = {Journal of Theoretical Biology},
  volume = {481},
  pages = {119--135},
  doi = {10.1016/j.jtbi.2018.09.008},
  langid = {english},
}

@article{ramirez2015,
  title = {Dendritic Spine Geometry Can Localize {{GTPase}} Signaling in Neurons},
  author = {Ramirez, Samuel A. and Raghavachari, Sridhar and Lew, Daniel J.},
  editor = {{Edelstein-Keshet}, Leah},
  year = {2015},
  month = nov,
  journal = {Molecular Biology of the Cell},
  volume = {26},
  number = {22},
  pages = {4171--4181},
  doi = {10.1091/mbc.E15-06-0405},
  langid = {english},
}

@article{giese2015,
  title = {Influence of Cell Shape, Inhomogeneities and Diffusion Barriers in Cell Polarization Models},
  author = {Giese, Wolfgang and Eigel, Martin and Westerheide, Sebastian and Engwer, Christian and Klipp, Edda},
  year = {2015},
  month = nov,
  journal = {Physical Biology},
  volume = {12},
  number = {6},
  pages = {066014},
  doi = {10.1088/1478-3975/12/6/066014},
  langid = {english},
}

@article{borgqvist2021,
  title = {Cell Polarisation in a Bulk-Surface Model Can Be Driven by Both Classic and Non-Classic {{Turing}} Instability},
  author = {Borgqvist, Johannes and Malik, Adam and Lundholm, Carl and Logg, Anders and Gerlee, Philip and Cvijovic, Marija},
  year = {2021},
  month = feb,
  journal = {npj Systems Biology and Applications},
  volume = {7},
  number = {1},
  pages = {13},
  publisher = {Nature Publishing Group},
  doi = {10.1038/s41540-021-00173-x},
  langid = {english},
}

@article{morita2020,
  title = {Turing Type Instability in a Diffusion Model with Mass Transport on the Boundary},
  author = {Morita, Yoshihisa and Sakamoto, Kunimochi},
  year = {2020},
  journal = {Discrete and Continuous Dynamical Systems},
  volume = {40},
  number = {6},
  pages = {3813--3836},
  doi = {10.3934/dcds.2020160},
  langid = {english},
}

@article{ratz2014,
  title = {Symmetry Breaking in a Bulk--Surface Reaction--Diffusion Model for Signalling Networks},
  author = {R{\"a}tz, Andreas and R{\"o}ger, Matthias},
  year = {2014},
  month = aug,
  journal = {Nonlinearity},
  volume = {27},
  number = {8},
  pages = {1805--1827},
  doi = {10.1088/0951-7715/27/8/1805},
  langid = {english},
}

@article{hausberg2018,
  title = {Well-Posedness and Fast-Diffusion Limit for a Bulk--Surface Reaction--Diffusion System},
  author = {Hausberg, Stephan and R{\"o}ger, Matthias},
  year = {2018},
  month = apr,
  journal = {Nonlinear Differential Equations and Applications NoDEA},
  volume = {25},
  number = {3},
  pages = {17},
  doi = {10.1007/s00030-018-0508-8},
  langid = {english},
}

@article{miller2023,
  title = {Generation and {{Motion}} of {{Interfaces}} in a {{Mass-Conserving Reaction-Diffusion System}}},
  author = {Miller, Pearson W. and Fortunato, Daniel and Novaga, Matteo and Shvartsman, Stanislav Y. and Muratov, Cyrill B.},
  year = {2023},
  month = sep,
  journal = {SIAM Journal on Applied Dynamical Systems},
  volume = {22},
  number = {3},
  pages = {2408--2431},
  doi = {10.1137/22M152548X},
  langid = {english},
}

@article{cusseddu2022,
  title = {Numerical Investigations of the Bulk-Surface Wave Pinning Model},
  author = {Cusseddu, Davide and Madzvamuse, Anotida},
  year = {2022},
  month = dec,
  journal = {Mathematical Biosciences},
  volume = {354},
  pages = {108925},
  doi = {10.1016/j.mbs.2022.108925},
  langid = {english},
}

@article{wei2001,
  title = {Spikes for the {{Two-Dimensional Gierer-Meinhardt System}}: {{The Weak Coupling Case}}},
  shorttitle = {Spikes for the {{Two-Dimensional Gierer-Meinhardt System}}},
  author = {Wei, J. and Winter, M.},
  year = {2001},
  month = dec,
  journal = {Journal of Nonlinear Science},
  volume = {11},
  number = {6},
  pages = {415--458},
  doi = {10.1007/s00332-001-0380-1},
  langid = {english},
}

@article{fife1973,
  title = {Semilinear Elliptic Boundary Value Problems with Small Parameters},
  author = {Fife, Paul C.},
  year = {1973},
  month = sep,
  journal = {Archive for Rational Mechanics and Analysis},
  volume = {52},
  number = {3},
  pages = {205--232},
  doi = {10.1007/BF00247733},
  langid = {english},
}

@article{ishii2021,
  title = {Concentration Phenomena on {{Y}}--Shaped Metric Graph for the {{Gierer}}--{{Meinhardt}} Model with Heterogeneity},
  author = {Ishii, Yuta},
  year = {2021},
  month = apr,
  journal = {Nonlinear Analysis},
  volume = {205},
  pages = {112220},
  doi = {10.1016/j.na.2020.112220},
  langid = {english},
}

@article{mori2011,
  title = {Asymptotic and {{Bifurcation Analysis}} of {{Wave-Pinning}} in a {{Reaction-Diffusion Model}} for {{Cell Polarization}}},
  author = {Mori, Yoichiro and Jilkine, Alexandra and Edelstein-Keshet, Leah},
  year = {2011},
  journal = {SIAM Journal on Applied Mathematics},
  volume = {71},
  number = {4},
  pages = {1401--1427},
  doi = {10.1137/10079118X},
  langid = {english},
}

@article{watanabe2026,
  title = {Geometry–Induced Pulse Dynamics in a Bulk–Surface Reaction–Diffusion System for Cell Polarization},
  author = {Watanabe, Riku},
  year = {2026},
  journal = {Journal of Mathematical Biology},
  shortjournal = {J. Math. Biol.},
  volume = {93},
  number = {5},
  pages = {53},
  doi = {10.1007/s00285-026-02473-4},
  langid = {english},
}
\end{document}